\documentclass[11pt]{article}

\usepackage{amsmath,amssymb,amsthm,latexsym}

\input epsf
\usepackage{epsfig}

\usepackage{color}
\usepackage{comment}
\allowdisplaybreaks

\newtheorem{theorem}{Theorem}[section]
\newtheorem{thm}[theorem]{Theorem}

\newtheorem{lemma}[theorem]{Lemma}
\newtheorem{proposition}[theorem]{Proposition}
\newtheorem{prop}[theorem]{Proposition}
\newtheorem{definition}[theorem]{Definition}
\newtheorem{remark}[theorem]{Remark}
\newtheorem{example}[theorem]{Example}
\numberwithin{equation}{section}

\def\be{\begin{equation}}
\def\ee{\end{equation}}
\def\bes{\begin{equation*}}
\def\ees{\end{equation*}}

\def\vp{{\varphi}}

\def\eps{\varepsilon}
\def\lam{{\lambda}}
\def\ol{\overline}

\def\qed{{\hfill $\square$ \bigskip}}

\def\esssup{{\mathop{\rm ess \; sup \, }}}
\def\essinf{{\mathop{\rm ess \; inf \, }}}

\def\sE {{\cal E}}
\def\sF {{\cal F}}
\def\sL {{\cal L}}
\def\sD{{\cal D}}

\def\sN {{\cal N}}

\def\bE {{\mathbb E}} \def\bN {{\mathbb N}}
\def\bP {{\mathbb P}} \def\bR {{\mathbb R}}

\def\dcap{\mathrm{cap}}
\def\EHR{\mathrm{EHR}}

\def\PHI{\mathrm{PHI}}
\def\UHK{\mathrm{UHK}}
\def\Gcap{\mathrm{Gcap}}

\def\HK{\mathrm{HK}}
\def\CSA{\mathrm{CSA}}
\def\CSJ{\mathrm{CSJ}}
\def\CS{\mathrm{CS}}

\def\PI{\mathrm{PI}}

\def\FK{\mathrm{FK}}
\def\NL{\mathrm{NL}}
\def\VD{\mathrm{VD}}
\def\RVD{\mathrm{RVD}}

\def\PHI{\mathrm{PHI}}
\def\EP{\mathrm{EP}}
\def\UHKD{\mathrm{UHKD}}
\definecolor{dred}{rgb}{0.8, 0.0, 0.0}

\def\NDL{\mathrm{NDL}}
\def\UJS{\mathrm{UJS}}

\def\E{\mathrm{E}}
\def\J{\mathrm{J}}
\def\<{\langle}
\def\>{\rangle}

\begin{document}

\title{\bf Heat kernel estimates for
general symmetric pure jump
Dirichlet forms
}

\author{{\bf Zhen-Qing Chen}, \quad {\bf Takashi Kumagai} \quad and \quad
{\bf Jian Wang}}

\date{}
\maketitle

\begin{abstract} In this paper,
we consider the following symmetric non-local Dirichlet forms of pure jump type on metric measure space $(M,d,\mu)$:
$$\sE(f,g)=\int_{M\times M} (f(x)-f(y))(g(x)-g(y))\,J(dx,dy),$$
where $J(dx,dy)$ is a symmetric Radon measure on $M\times M\setminus {\rm diag}$ that
may have
 different scalings for small jumps and large jumps.
Under general volume doubling condition on $(M,d,\mu)$ and some mild
quantitative assumptions on $J(dx, dy)$ that are
allowed to have light tails
of polynomial decay at infinity, we
establish stability results for
two-sided heat kernel estimates as well as heat kernel upper bound estimates
in terms of jumping kernel bounds,
the cut-off Sobolev inequalities,
and the Faber-Krahn inequalities (resp.\ the
Poincar\'e inequalities).
We also give
stable characterizations of the corresponding parabolic Harnack inequalities.
\end{abstract}



{\small
 \begin{tableofcontents}
\end{tableofcontents} }

\section{Introduction and main results} \label{S:1}

\subsection{Motivating example}\label{S:1.1}

\quad This paper can be regarded as a sequel to \cite{CKW1, CKW2} and \cite{CKW4}, where stability of heat kernel estimates and parabolic Harnack inequalities
are
studied for symmetric pure jump Markov processes and for symmetric diffusions with jumps. In these papers, the scale functions for large jumps are all assumed to be less than that for the diffusion processes if there is one
(for example, $r\mapsto r^2$ in the Euclidean space case). In this paper, we consider
general
symmetric pure jump Markov processes
whose jumping kernel
can
have light tails at infinity.
New phenomena arise in the light
tail case. To illustrate some of the new phenomena, we start with an example.

\begin{example}\label{Exam-M}\rm
Let $(M,d)$ be a locally compact separable metric space,
and $\mu$
a positive Radon measure on $M$ with full support.
We will refer to such a triple $(M, d, \mu)$ as a \emph{metric measure space}. Let $V(x,r)=\mu(B(x,r))$ for any $x\in M$ and $r>0$. Assume that the volume doubling (VD) condition holds, i.e., there is a constant $c_\mu>0$ so that
$$
V(x,2r)\le c_\mu V(x,r),\quad x\in M,r>0.
$$
Suppose $X:=\{X_t, {t\ge0}; \bP^x, {x\in M}\}$ is a
conservative symmetric strong Markov process on $M$ that has a
transition density function $q(t,x,y)$ with respect to $\mu$ such
that
\begin{equation}\label{Up}
 q(t,x,y)\asymp
 \frac{1}{V(x,t^{1/\beta})}\exp\left(-c\left(\frac{d(x,y)^\beta}{t}\right)^{1/(1-\beta)}\right),
 \quad t>0, x,y\in M
 \end{equation}
for some $\beta \geq 2$. Here and in what follows, for two positive
functions $f(t, z)$ and $g(t, z)$,
 notation $f\asymp g$ means that there exist positive constants $c_i$ $(i=1,\cdots, 4)$ such that
$$
c_1f(c_2 t, z)\le g(t, z)\le c_3f(c_4 t, z).
$$
Let $S:=(S_t)_{t\ge0}$ be a subordinator with $S_0=0$
that is independent of $X$ and has the Laplace exponent
$$f(r)=\int_0^\infty (1-e^{-rs})\nu(s)\,ds,\quad r>0,$$ where
$$\nu(s)=\frac{1}{s^{1+\gamma_1}}{\bf 1}_{\{0<s\le 1\}}+ \frac{1}{s^{1+\gamma_2}}{\bf 1}_{\{s>1\}}$$ with $\gamma_1\in (0,1)$ and $\gamma_2\in(1,\infty)$.
Let $Y:=(Y_t)_{t\ge0}$ be the
subordinate process
 defined by
$Y_t:=X_{S_t}$ for all $t>0$. For a set $A\subset M$, define the
exit time $\tau^Y_A = \inf\{ t >0 : Y_t \notin A\}.$ The subordinate
process $Y$ has the following properties.

\medskip

{\bf Claim 1.}~
$Y$ is a symmetric jump process such that
\begin{itemize}
\item[{\rm(i)}]
its jumping kernel $J(dx, dy)$ has a density with respect to the
product measure $\mu\times \mu$ on $M\times M\setminus {\rm diag}$
given by
\begin{equation}\label{eq:jownex1}
J(x,y)\simeq\begin{cases}
\frac{1}{V(x,d(x,y))d(x,y)^{\alpha_1}},&\quad d(x,y)\le 1,\\
\frac{1}{V(x,d(x,y))d(x,y)^{{\alpha_2}}},&\quad d(x,y)\ge 1,
\end{cases}
\end{equation}
where ${\alpha_1}=\gamma_1\beta$ and ${\alpha_2}=\gamma_2\beta$, and {\rm diag} stands for the diagonal of $M\times M$.

\item[{\rm(ii)}] for any $x_0\in M$ and $r>0$,
$$
\bE^x \left[ \tau^Y_{B(x_0,r)} \right] \simeq r^{{\alpha_1}} \vee
r^{\beta}
$$
and
$$
{\rm cap}_Y
(B(x_0,r), B(x_0,2r))\simeq\begin{cases} {V(x_0,r)}/{r^{\alpha_1}}&\quad \hbox{for } r\le 1,\\
 {V(x_0,r)}/{r^{\beta}}&\quad \hbox{for } r> 1.
\end{cases}
$$
\end{itemize}
Here,
for any subsets $A\subset B$, the relative capacity
$\dcap_Y(A,B)$ is defined as
$$\dcap_Y(A,B):=\inf \left\{\sE_Y(\vp,\vp): \vp
=1 \ \sE_Y \hbox{-q.e. on $A$ and } \vp =0 \ \sE_Y \hbox{-q.e. on } B^c  \right\},
$$
where $\sE_Y(\cdot,\cdot)$ is the non-local Dirichlet form associated with the process $Y$.
In this paper, we use $:=$ as a way of definition.
We write $f\simeq g$, if there exist constants $c_1,c_2>0$ such that $c_1g(x)\le f(x)\le c_2g(x)$ for the specified range of $x$.
Similarly, we write $f\preceq g$ if there exists a constant $c >0$ such that $ f(x)\le c g(x)$.
For real numbers $a$ and $b$,
set $a\vee b:
=\max \{a, b\}$ and $a\wedge b:=\min\{a, b\}$.

\medskip

{\bf Claim 2.}
The process $Y$ has a jointly continuous transition density function $p(t, x, y)$ with respect to
the measure $\mu$ on $M$ so that
\begin{equation}\label{e:1.3}
p(t, x, y) \simeq \frac{1}{V(x,t^{1/{\alpha_1}})}\wedge (t J(x,y)) \quad \hbox{for } t\leq 1
\end{equation}
and
\begin{equation}\label{e:1.4}\begin{split}
&c_1 \left(  \frac{1}{V(x,t^{1/\beta})}{\bf1}_{\{d(x, y) \leq t^{1/\beta}\}}
+ \frac{t}{V(x, d(x, y))d(x, y)^{\alpha_2}} {\bf1}_{\{d(x, y) > t^{1/\beta}\}} \right) \nonumber\\
& \leq  p(t,x,y)\leq  c_2 \left( \frac{1}{V(x,t^{1/\beta})}\wedge \left(tJ(x,y)+
q(c_3t,x,y)\right) \right)
\quad \hbox{for } t\ge 1,
\end{split}\end{equation}
where $q(t,x, y)$ is the transition density function for the diffusion process $X$
 of the form \eqref{Up}.
If, in addition, $(M,d,\mu)$ is connected and satisfies the chain condition, then
\begin{equation}\label{e:1.5}
p(t,x,y)\asymp \begin{cases}\frac{1}{V(x,t^{1/{\alpha_1}})}\wedge (t J(x,y)),&\quad t\le 1,\\
\frac{1}{V(x,t^{1/\beta})}\wedge \left(tJ(x,y)+q(t,x,y)\right),&\quad t\ge 1.
\end{cases}
\end{equation}
 \end{example}

The proof of
Example \ref{Exam-M} will be given in the appendix at the end of
this paper. The first assertion (i)  of Claim 1
says that the
scaling function of jumping kernel for the process $Y$ is
$\phi_j(r):=r^{{\alpha_1}}\vee r^{{\alpha_2}}$ with ${\alpha_1}<\beta<{\alpha_2}$,
while the second assertion (ii)  of Claim 1
indicates that the scaling function
of the process $Y$ is $r^{{\alpha_1}}\vee r^{\beta}$. The scaling
function of the process $Y$ is different from the associated jumping
kernel at large scale (that is, when $r>1$). Thus the behavior of
the symmetric jump process $Y$ in Example \ref{Exam-M} is different
from that of symmetric ${\alpha_1}$-stable-like or mixed stable-like
processes studied in \cite{CK1,CKW1}. This may appear surprising at
the first glance that the process $Y$ enjoys the diffusive scaling
$\phi_c(r):=r^\beta$ when $r>1$. But it becomes
quite reasonable if one
thinks more about it as the jumping kernel $J(dx,dy)$ of $Y$ has finite
second moment in the case of $\beta=2$. Note that in this example,
$\phi_j(r) \geq \phi_c(r)$ for all $r>0$.

 The goal of this paper is to
 study heat kernel estimates and their applications
 systematically for
 general pure jump Dirichlet forms, which include
 a large class
of symmetric pure jump Dirichlet forms that
have light jumping kernel
 at infinity
  and thus possibly exhibit diffusive behaviors.
 The main results of this paper are Theorems \ref{T:main}, \ref{T:1.13} and \ref{T:PHI}.

\subsection{Setting}\label{S:1.2}

Let $(M,d,\mu)$ be a measure metric space.
Throughout the paper, we
assume that all balls are relatively compact, and
assume for simplicity that $\mu (M)=\infty$.
 We emphasize that in this paper we do not assume $M$ to be connected nor
 $(M, d)$ to be geodesic.
We are concerned with
regular Dirichlet
forms $(\sE, \sF)$
on $L^2(M; \mu)$ that only have  pure-jump part; namely,
$$
\sE(f,g)=\int_{M\times M\setminus \textrm{diag}}(f(x)-f(y)(g(x)-g(y))\,J(dx,dy),\quad f,g\in \sF,
$$
where $J(\cdot,\cdot)$ is a symmetric Radon measure on $M\times
M\setminus \textrm{diag}$.

Throughout the paper,
we will always take a quasi-continuous version of $f\in \sF$ without denoting it by $\tilde f$.
Let $(\sL, \sD (\sL))$
be the (negative definite) $L^2$-\emph{generator} of $(\sE,
\sF)$ on $L^2(M; \mu)$, and let $\{P_t\}_{t\geq 0}$ be
the associated \emph{semigroup}. Associated with the regular
Dirichlet form $(\sE, \sF)$ on $L^2(M; \mu)$ is a $\mu$-symmetric
\emph{Hunt process} $X=\{X_t, t \ge 0; \, \bP^x, x \in M\setminus
\sN\}$. Here
$\sN \subset M$
is a properly exceptional set for $(\sE, \sF)$ in
the sense that
 $\sN$ is nearly Borel,
 $\mu(\sN)=0$ and $M_\partial\setminus \sN$ is $X$-invariant.
This Hunt process is unique up to a
properly exceptional set --- see \cite[Theorem 4.2.8]{FOT}.
We fix $X$ and $\sN$, and write $ M_0 = M \setminus \sN.$
The \emph{heat kernel} associated with the semigroup $\{P_t\}_{t\ge0}$ (if it exists) is a
measurable function
$p(t, x,y):M_0 \times M_0 \to (0,\infty)$ for every $t>0$, such that
\begin{align} \label{e:hkdef}
\bE^x f(X_t) &= P_tf(x) = \int p(t, x,y) f(y) \, \mu(dy)\quad \hbox{for all } x \in M_0,
f \in L^\infty(M;\mu), \\
p(t, x,y) &= p(t, y,x) \quad \hbox{for all } t>0,\, x ,y \in M_0, \label{e:1.7} \\
\label{e:ck}
p(s+t , x,z) &= \int p(s, x,y) p(t, y,z) \,\mu(dy)
\quad \hbox{for all } s>0, t>0,
\,\, x,z \in M_0.
\end{align}
While \eqref{e:hkdef} only determines $p(t, x, \cdot)$ $\mu$-a.e.,
using the Chapman-Kolmogorov equation \eqref{e:ck} one can
regularize $p(t, x,y)$ so that \eqref{e:hkdef}--\eqref{e:ck} hold
for every point in $M_0$. See \cite[Theorem 3.1]{BBCK} and
\cite[Section 2.2]{GT} for details. We call $p(t, x,y)$ the {\em
heat kernel} on the \emph{metric measure Dirichlet space} (or
\emph{MMD space}) $(M, d, \mu, \sE)$.
By \eqref{e:hkdef}, sometime we also call $p(t,x,y)$ the \emph{transition density function} with respect to the measure $\mu$ for the process $X$.

\begin{definition} \label{D:1.2} {\rm Denote the ball centered at $x$ with radius $r$ by
$B(x, r)$ and $\mu (B(x, r))$ by $V(x, r)$.

\noindent (i) We say that $(M,d,\mu)$ satisfies the {\it volume
doubling property} ($\VD$), if there exist constants $d_2,{C_\mu}>0$ such that for all $x \in M$ and $r > 0$, \be \label{e:vd2}
\frac{V(x,R)}{V(x,r)} \leq {C_\mu} \Big(\frac Rr\Big)^{d_2} \quad \hbox{for all } x\in M \hbox{ and }
0<r\le R.
\ee
(ii) We say that $(M,d,\mu)$ satisfies the {\it reverse volume doubling property} ($\RVD$),
if there exist constants
$d_1,{c_\mu} >0$ such that
for all $x \in M$ and $0<r\le R$,
\be \label{e:rvd}
\frac{V(x,R)}{V(x,r)}\ge {c_\mu} \Big(\frac Rr\Big)^{d_1}.
\ee
 }\end{definition}

Note that VD condition \eqref{e:vd2} is equivalent to the existence of a positive constant $c_1$ so that
$$
V(x, 2r) \leq c_1 V(x, r) \quad \hbox{for all } x\in M \hbox{ and } r>0,
$$
while RVD condition \eqref{e:rvd} is equivalent to the existence of positive constants
$c_2>1$ and $l_0>0$ so that
$$
V(x, l_0 r) \geq c_2 V(x, r) \quad \hbox{for all } x\in M \hbox{ and } r>0.
$$

It is known that $\VD$ implies $\RVD$ if $M$ is connected and
unbounded; e.g.\ see \cite[Proposition 5.1 and Corollary 5.3]{GH0}.
In this paper, we always assume that $\VD$ holds, and sometime we
also assume $\RVD$ holds.

Set $\bR_+:=[0,\infty)$. Let $\phi_j: \bR_+\to \bR_+$ and $\phi_c: [1,\infty)\to \bR_+$
be strictly increasing continuous
functions with
$\phi_j (0)=\phi_c(0)=0$,  $\phi_j(1)=\phi_c(1)=1$
and satisfying that there exist constants
$c_{1,\phi_j},c_{2,\phi_j}, c_{1,\phi_c}, c_{2,\phi_c}>0$,
 $\beta_{2,\phi_j}\ge \beta_{1,\phi_j}>0$  and $\beta_{2,\phi_c}\ge \beta_{1,\phi_c}>1$
such that
\begin{equation}\label{polycon}
\begin{split}
 &c_{1,\phi_j} \Big(\frac Rr\Big)^{\beta_{1,\phi_j}} \leq
\frac{\phi_j (R)}{\phi_j (r)} \ \leq \ c_{2,\phi_j} \Big(\frac
Rr\Big)^{\beta_{2,\phi_j}}
\quad \hbox{for all }
0< r \le R,  \\
  &c_{1,\phi_c} \Big(\frac Rr\Big)^{\beta_{1,\phi_c}} \leq
\frac{\phi_c (R)}{\phi_c (r)} \ \leq \ c_{2,\phi_c} \Big(\frac
Rr\Big)^{\beta_{2,\phi_c}}
\quad \hbox{for all }  0< r \le R.
\end{split}
\end{equation}

Since $\beta_{1,\phi_c}>1$, we know from \cite[Definition, p.\ 65; Definition, p.\ 66; Theorem 2.2.4 and its remark, p.\ 73]{BGT}
that there exists a strictly increasing function
$\bar \phi_c : \bR_+\to \bR_+$
such that there is a constant $c_1\ge1$ so that
\begin{equation}\label{e:1.10}
 c_1^{-1}\phi_c(r)/r \le \bar \phi_c (r)\le c_1 \phi_c(r)/r \quad \hbox{for all }
r>0.
\end{equation}
Clearly, by
\eqref{polycon} and \eqref{e:1.10}, there exist constants $c_{1,\bar\phi_c}, c_{2,\bar\phi_c}>0$ such that
\begin{equation}\label{e:effdiff}c_{1,\bar\phi_c} \Big(\frac Rr\Big)^{\beta_{1,\phi_c}-1} \leq
\frac{\bar\phi_c (R)}{\bar\phi_c (r)} \ \leq \ c_{2,\bar\phi_c} \Big(\frac
Rr\Big)^{\beta_{2,\phi_c}-1}
\quad \hbox{for all }
0<  r \le R.\end{equation}

Define
$$
\beta_*:= \sup\left\{\beta>0: \hbox{there is a constant } c_*>0 \hbox{ so that } \phi_j(R)/\phi_j (r)
\geq c_* (R/r)^\beta \hbox{ for   } 0<r<R\leq 1 \right\},
$$
and
$$
\beta^*:= \sup\left\{\beta>0: \hbox{there is a contant } c^*>0 \hbox{ so that } \phi_j(R)/\phi_j (r)
\geq c^* (R/r)^\beta \hbox{ for all } R>r\geq 1 \right\}.
$$
Throughout this paper, we assume that there is a constant $c_0\ge 1$ so that
\begin{equation}\label{e:1.9}
\phi_c(r)\le c_0\phi_j(r) \ \hbox{ on } [0, 1] \hbox{ if } \beta_*>1
\quad \hbox{and} \quad  \phi_c(r)\le c_0\phi_j(r) \  \hbox{on } (1, \infty)
\hbox{ if } \beta^*>1.
\end{equation}
We point out that, by \eqref{polycon}
with
$\beta_{1,\phi_c}>1$,
 $\phi_c$ is not comparable to $\phi_j$ on $[0, 1]$
when $\beta_*\leq 1$, and $\phi_c$ is not comparable to $\phi_j$ on $[1, \infty )$
when $\beta^*\leq 1$.

In this paper,
the function $\phi_j$ plays the role of scaling function in
the jumping kernel; while $\phi_c$ is a scale function that should be intrinsically
determined by $\phi_j$ and
the metric measure space $(M, d, \mu)$,
which
will possibly appear in the expression of heat kernel estimates.
However, we do not have a universal formula for $\phi_c$.
For example, when the state space $M$ (such as $\bR^d$ or a nice fractal) has a nice diffusion process,  $\phi_c$ can be the scaling function of the diffusion in some cases
but can also be a different scale function in some other cases;
see Examples \ref{Exam-M} and \ref{e:example5.2}.
In some cases, $\phi_c$ can just be $\phi_j$ on a part or the whole of $[0, \infty)$.
To cover a wide spectrum of scenarios, in the formulation and characterization we allow
$\phi_c$ to be any
function that satisfies
conditions \eqref{polycon} and \eqref{e:1.9}.

Given $\phi_c$ and $\phi_j$ as above, we set
\begin{equation}\label{e:1.13}
\phi(r) :=\begin{cases}
\phi_j(r){\bf 1}_{\{ \beta_*\le 1\}}+ \phi_c(r){\bf 1}_{\{ \beta_*>1\}}
\qquad \hbox{for } 0<r\leq 1\\
 \phi_j(r){\bf 1}_{\{ \beta^*\le 1\}}+ \phi_c(r){\bf 1}_{\{ \beta^*>1\}}
\qquad \hbox{for } r>1.
\end{cases}
\end{equation}
In view of the assumptions above,
$\phi$ is strictly increasing on $\bR_+$ such that $\phi(0)=0$, $\phi(1)=1$ and there exist constants $c_{1,\phi},c_{2,\phi}>0$ so that
\begin{equation}\label{polycon000} c_{1,\phi} \Big(\frac Rr\Big)^{\beta_{1,\phi}} \leq
\frac{\phi(R)}{\phi (r)} \ \leq \ c_{2,\phi} \Big(\frac
Rr\Big)^{\beta_{2,\phi}}
\quad \hbox{for all }
0<r \le R,\end{equation} where
$\beta_{1,\phi}=\beta_{1,\phi_c}\wedge \beta_{1,\phi_j}$ and $\beta_{2,\phi}=\beta_{2,\phi_c}\vee\beta_{2,\phi_j}$.
Clearly, we have by \eqref{e:1.9} that
 $$
 \phi(r)\le c_0\phi_j(r) \quad \hbox{for every } r\ge0.
 $$
Denote by $\phi_j^{-1}(t)$, $\phi_c^{-1}(t)$ and $\bar\phi_c^{-1}(t)$  the inverse functions of the strictly
increasing functions $t\mapsto \phi_j(t)$, $t\mapsto \phi_c (t)$ and $t\mapsto \bar\phi_c (t)$, respectively.  Then the inverse function $\phi^{-1}$ of $\phi$ is given by
$$
\phi^{-1}(r) :=\begin{cases}  \phi^{-1}_j(r){\bf 1}_{\{ \beta_*\le 1\}}+ \phi^{-1}_c(r){\bf 1}_{\{ \beta_*>1\}}
\qquad \hbox{for } 0<r\leq 1 ,   \\
 \phi^{-1}_j(r){\bf 1}_{\{ \beta^*\le 1\}}+ \phi^{-1}_c(r){\bf 1}_{\{ \beta^*>1\}}
\qquad \hbox{for } r>1.
\end{cases}$$
 Throughout this paper,
we will fix the notations for these functions $\phi_c$, $\phi_j$, $\phi$ and $\bar \phi_c$.

\begin{definition}{\rm Let $\psi:\bR_+\to \bR_+$. We say $\J_{\psi}$ holds, if there exists a non-negative symmetric
function $J(x, y)$ so that for all $x, y \in M$,
\begin{equation}\label{e:1.15}
J(dx,dy)=J(x, y)\,\mu(dx)\, \mu (dy),
\end{equation} and
\begin{equation}\label{jsigm}
 \frac{c_1}{V(x,d(x, y)) \psi (d(x, y))}\le J(x, y) \le \frac{c_2}{V(x,d(x, y)) \psi (d(x, y))}.
 \end{equation}
We say that $\J_{\psi,\le}$ (resp. $\J_{\psi,\ge}$) if \eqref{e:1.15} holds and the upper bound (resp. lower bound) in \eqref{jsigm} holds.}
\end{definition}

Note that, since
$\phi(r)\le c_0 \phi_j(r$) for all $r>0$,
$\J_{\phi_j,\le}$ implies $\J_{\phi,\le}$; that is, $\J_{\phi,\le}$ is weaker than $\J_{\phi_j,\le}$.

\subsection{Main results }\label{S:1.3}
To state our results about upper bounds and two-sided estimates
on heat kernel precisely, we need a number of definitions. Some of
them are taken from \cite{CKW1}.

 For $f, g\in \sF$, we define the carr\'e du-Champ operator $\Gamma (f, g)$ for the
 pure jump
 Dirichlet form
$(\sE, \sF)$ by
$$
\Gamma (f, g) (dx) = \int_{y\in M} (f(x)-f(y))(g(x)-g(y))\,J(d x,d
y) .
$$
Clearly $\sE(f,g)= \Gamma(f,g) (M)$.

Let $U \subset V$ be open sets of $M$ with $U \subset \ol U \subset
V$, and let $\kappa\ge 1$. We say a non-negative bounded measurable
function $\vp$ is a {\it $\kappa$-cut-off function for $U \subset
V$}, if $\vp \ge 1$ on $U$, $\vp=0$ on $V^c$ and $0\leq \vp \leq
\kappa$ on $M$. Any 1-cut-off function is simply referred to as a
{\it cut-off function}.

 We now introduce the following cut-off Sobolev inequality
 $\CSJ(\phi)$ that controls the energy of cut-off functions.

\begin{definition} \rm Let $\sF_b= \sF\cap L^\infty(M,\mu)$.
We say that condition $\CSJ(\phi)$ holds, if there exist constants
$C_0\in (0,1]$ and $C_1, C_2>0$ such that for every $0<r\le R$,
almost all $x_0\in M$ and any $f\in \sF$, there exists a cut-off
function $\vp\in \sF_b$ for $B(x_0,R) \subset B(x_0,R+r)$ such that
\be \label{e:csj1} \begin{split}
 & \int_{B(x_0,R+(1+C_0)r)} f^2 \, d\Gamma (\vp,\vp) \\
& \, \le C_1 \int_{U\times U^*}(f(x)-f(y))^2\,J(dx,dy)
 + \frac{C_2}{\phi(r)} \int_{B(x_0,R+(1+C_0)r)} f^2 \,d\mu,
\end{split}
\ee
 where $U=B(x_0,R+r)\setminus B(x_0,R)$ and
$U^*=B(x_0,R+(1+C_0)r)\setminus B(x_0,R-C_0r)$.
\end{definition}

\begin{remark}\label{rek:scj}\rm
\begin{itemize}
\item[(i)] $\CSJ(\phi)$
for pure jump Dirichlet forms is first introduced in \cite{CKW1} as
a counterpart
 of $\CSA(\phi)$ for strongly
local Dirichlet forms (see \cite{AB,BB2,BBK1}). As pointed out in
\cite[Remark 1.6(ii)]{CKW1}, the main difference between
$\CSJ(\phi)$ and $\CSA(\phi)$ is that the integrals in the left hand
side and in the second term of the right hand side of the inequality
\eqref{e:csj1} are over $B(x_0,R+(1+C_0)r)$ instead of over
$B(x_0,R+r)$ in \cite{AB}. Note that the integral over
$B(x_0,R+r)^c$ is zero in the left hand side of \eqref{e:csj1} for
the case of strongly local Dirichlet forms. As we see from the
approach of \cite{CKW1} in the study of stability of heat kernel
estimates for symmetric mixed stable-like processes, it is important
to enlarge the ball $B(x_0,R+r)$ and integrate over
$B(x_0,R+(1+C_0)r)$ rather than over $B(x_0,R+r)$.

 \item[(ii)] Denote by $\sF_{loc}$ the space of functions
locally in $\sF$; that is, $f\in \sF_{loc}$ if and only if for any
relatively compact open set $U\subset M$ there exists $g\in \sF$
such that $f=g$ $\mu$-a.e.\ on $U$.
 Since each ball is relatively compact and
 \eqref{e:csj1} uses the property of $f$ on $B(x_0,R+(1+C_0)r)$ only, $\CSJ(\phi)$ also holds for any $f\in \sF_{loc}.$

\item[(iii)] Consider the non-local symmetric Dirichlet form $(\sE,\sF)$ on any geodesically complete Riemannian manifold
that satisfies $\VD$ and $\PI (r^2)$ conditions
so that
$\J_{\phi_j,\le}$ holds, where $\phi_j(r) :=r^\alpha\vee r^\beta$ and $0<\alpha<2<\beta$.
Then, by a similar argument as that for Example \ref{Exam-M} given in the Appendix of this paper,
$\CSJ(\phi)$ holds
 with $\phi(r)=r^\alpha \vee r^2$.
\end{itemize}
\end{remark}

Following \cite[Definition 1.11]{GHH}, we can adopt the generalized capacity inequality to the present setting, whose definition is much simpler than $\CSJ(\phi)$.

\begin{definition}\label{D:gcap}\rm We say that the {\it generalized capacity inequality} $\Gcap(\phi)$ holds, if there exist constants $\kappa\ge1$ and $C>0$ such that
for every $0<r<R$, $x_0\in M$ and any $u\in \sF'_b:=\{u+a: u\in \sF_b, \, a\in \bR\}$, there is a
$\kappa$-cut-off function $\vp\in \sF$ for $B(x_0,R)\subset
B(x_0,R+r)$ so that
$$\sE(u^2\vp,\vp)\le \frac{C}{\phi(r)} \int_{B(x_0,R+r)}u^2\,d\mu.$$
\end{definition}

We next introduce the Faber-Krahn inequality, see \cite[Section
3.3]{GT} for more details. For $\alpha >0$, we define
$$
\sE_\alpha (f, g)=\sE(f, g)+ \alpha  \int_M f(x)g(x)\,\mu (dx) \quad \hbox{for }
f, g\in \sF.
$$
For any open set $D \subset M$, $\sF_D$ is defined to be the
$\sE_1$-closure in $\sF$ of $\sF\cap C_c(D)$.
Define
\be \label{e:lam1}
 \lam_1(D)
= \inf \left\{ \sE(f,f): \, f \in \sF_D \hbox{ with } \|f\|_2 =1 \right\},
\ee
the bottom of the Dirichlet spectrum of $-\sL$ on $D$.

\begin{definition}\label{Faber-Krahnww}
{\rm The MMD space $(M,d,\mu,\sE)$ satisfies the {\em Faber-Krahn
inequality} $\FK(\phi)$, if there exist positive constants $C$ and
$p$ such that for any ball $B(x,r)$ and any open set $D \subset
B(x,r)$, \be \label{e:fki}
 \lam_1 (D) \ge \frac{C}{\phi(r)} (V(x,r)/\mu(D))^p.
\ee
} \end{definition}

\begin{definition} {\rm
We say that
the {\em $($weak$)$ Poincar\'e inequality}
$\PI(\phi)$
holds if there exist constants $C>0 $ and $\kappa\ge1$ such that
for any ball $B_r=B(x,r)$ with $x\in M$ and for any $f \in \sF_b$,
\begin{equation}\label{eq:PIn}
\int_{B_r} (f-\ol f_{B_r})^2\, d\mu \le C \phi(r)\int_{{B_{\kappa
r}}\times {B_{\kappa r}}} (f(y)-f(x))^2\,J(dx,dy),
\end{equation}
where $\ol f_{B_r}= \frac{1}{\mu({B_r})}\int_{B_r} f\,d\mu$ is the average value of $f$ on ${B_r}$.} \end{definition}

Recall that $X=\{X_t\}$ is the Hunt process associated with the regular Dirichlet form $(\sE,\sF)$ on $L^2(M; \mu)$
with properly exceptional set $\sN$, and $M_0:= M\setminus \sN$.
For a set
$A\subset M$, define the exit time $\tau_A = \inf\{ t >0 : X_t
\notin A\}.$
We need the following definition
on the exit time from balls.

\begin{definition}{\rm (i) We say that $\E_\phi$ holds if
there is a constant $c_1>1$ such that for all $r>0$ and all $x\in M_0$,
$$c_1^{-1}\phi(r)\le \bE^x [ \tau_{B(x,r)} ] \le c_1\phi(r).$$ We say that $\E_{\phi,\le}$ (resp. $\E_{\phi,\ge}$) holds
if the upper bound (resp. lower bound) in the inequality above
holds.

(ii) We say $\EP_{\phi,\le}$ holds if there is a constant $c>0$ such
that for all $r,t>0$ and all $x\in M_0$,
 $$
 \bP^x \left( \tau_{B(x,r)} \le t \right) \le \frac{ct}{\phi(r)}.
 $$}
\end{definition}

Recall that the function $\bar \phi_c (r)$ is a strictly increasing function satisfying \eqref{e:1.10}.
For any $t>0$ and $x,y\in M_0$, set
\begin{equation}\label{e:1.26}
p^{(j)}(t, x, y):= \frac1{V(x, \phi_j^{-1}(t))}\wedge \frac{t}{V(x,
d(x, y))\phi_j (d(x, y))}
\end{equation} and
\begin{equation}\label{e:1.24}
p^{(c)}(t,x,y):= \frac{1}{V(x, \phi_c^{-1}(t))}
 \exp\left(-\frac{d(x, y)}{\bar \phi_c^{-1}(t/d(x, y))} \right).
 \end{equation}

\begin{definition}\label{D:1.11} \rm \begin{description}
\item{\rm (i)} We say that $\HK(\phi_j,\phi_c)$ holds if there exists a heat kernel $p(t, x,y)$
of the semigroup $\{P_t\}$ associated with $(\sE,\sF)$,
 which has the
following estimates for  all $x,y\in M_0$,
\begin{equation}\label{HKjum}
   \ p(t, x,y) \asymp
 \begin{cases}
\begin{cases} p^{(j)}(t,x,y)&\,\,
\hbox{if }  \beta_*\le 1\\
 \frac 1{V(x,\phi_c^{-1}(t))} \wedge\big(p^{(c)}( t,x,y)+p^{(j)}(t,x,y)\big)\Big)&\,\,
\hbox{if } \beta_*>1\end{cases}\quad \hbox{for  }0<t\le 1 \\
  \begin{cases} p^{(j)}(t,x,y)&\,\, \hbox{if } \beta^*\le 1\\
 \frac 1{V(x,\phi_c^{-1}(t))} \wedge\big(p^{(c)}(t,x,y)+p^{(j)}(t,x,y)\big)\Big)
&\,\, \hbox{if } \beta^*>1\end{cases}\quad \hbox{for  }t> 1.
\end{cases}
\end{equation}

\item{(ii)} We say $\HK_-(\phi_j,\phi_c)$ holds if the upper bound in \eqref{HKjum} holds but the lower bound is replaced
by the following statement: there are constants $c_1, c_2>0$ so that
$$
p(t, x, y) \geq
c_1\begin{cases} \frac 1{V(x,\phi^{-1}(t))},&\quad d(x,y)\le c_2\phi^{-1}(t),\\
\frac{t}{V(x,d(x,y))\phi_j(d(x,y))} ,&\quad d(x,y)\ge c_2\phi^{-1}(t).
\end{cases}
$$

\item{(iii)} We say $\UHK(\phi_j,\phi_c)$
holds if the upper bound
in \eqref{HKjum} holds.

\item{(iv)} We say $\UHKD(\phi)$ holds if there is a constant $c>0$ such that for all $t>0$ and all $x\in M_0$,
$$p(t, x,x)\le \frac{c}{V(x,\phi^{-1}(t))}.$$

\item{(v)} We say $\NL(\phi)$ holds if there are constants $c_1,c_2>0$ such that for all $t>0$ and all $x,y\in M_0$ with $d(x,y)\le
c_1\phi^{-1}(t)$,
$$p(t, x,y)\ge \frac{c_2}{V(x,\phi^{-1}(t))}.$$

\item{(vi)} For any
open subset $D \subset M$, denote
 by $(P_t^D)_{t\ge0}$ the (Dirichlet) semigroups of $(\sE,\sF_D)$, and by $p^D(t,x,y)$ the corresponding (Dirichlet) heat kernel.
We say that \emph{a
near diagonal lower bounded estimate for Dirichlet heat kernel} $\NDL(\phi)$ holds, if there exist $\eps\in (0,1)$ and $c_1>0$ such that for any $x_0\in M$, $r>0$
and $0<t\le \phi(\eps r)$,
\be\label{NDLdf1}
p^{B(x_0, r)}(t, x ,y )
\ge \frac{c_1}{V(x_0, \phi^{-1}(t))}  \qquad \hbox{for } x ,y\in B(x_0,
\eps\phi^{-1}(t)) \cap M_0.
\ee

\end{description}
\end{definition}

It is obvious that $\NDL(\phi)$ is stronger than $\NL(\phi)$,
and $\HK_- (\phi_j, \phi_c)$ implies
$\NL (\phi)$. On the other hand, under $\VD$ and \eqref{polycon}, we can verify that $\NL(\phi)$ together with $\UHK(\phi_j,\phi_c)$ implies $\NDL(\phi)$;
see the proof of \cite[Lemma 5.7]{CKW4}.

\medskip

We postpone  discussions on the formulations of $\HK_- (\phi_j, \phi_c)$ and
$\HK (\phi_j, \phi_c)$ to Subsection \ref{S:1.4}. The following are the two main results
of this paper on heat kernel estimates.

\begin{thm} \label{T:main}
Assume that the metric measure space $(M, d , \mu)$ satisfies $\VD$
and $\RVD$, and
the functions $\phi_c$ and $\phi_j$ satisfy \eqref{polycon} and \eqref{e:1.9}.
Let $\phi (r)$ be defined by \eqref{e:1.13}.
The following are equivalent. \smallskip \\
$(1)$ $\HK_-(\phi_j,\phi_c)$. \\
$(2)$ $\UHK(\phi_j,\phi_c)$, $\NL(\phi)$ and $\J_{\phi_j}$. \smallskip \\
$(3)$ $\UHKD(\phi)$, $\NDL(\phi)$ and $\J_{\phi_j}$. \smallskip \\
$(4)$ $\PI(\phi)$, $\J_{\phi_j}$ and $\E_\phi$. \smallskip \\
$(5)$ $\PI(\phi)$, $\J_{\phi_j}$ and $\Gcap(\phi)$. \smallskip \\
$(6)$ $\PI(\phi)$, $\J_{\phi_j}$ and $\CSJ(\phi)$. \smallskip \\
If, in additional, $(M,d,\mu)$ is connected and satisfies the chain
condition, then
each above assertion is equivalent to \smallskip \\
$(7)$ $\HK(\phi_j,\phi_c)$.
\end{thm}

In the process of establishing Theorem \ref{T:main}, we also obtain the following characterizations for $\UHK(\phi_j,\phi_c)$. In the following, we say
{\it $(\sE, \sF)$ is conservative} if its associated Hunt process $X$
has infinite lifetime. This is equivalent to $P_t 1 =1$
on $M_0$ for every $t>0$.
It follows from \cite[Proposition 3.1]{CKW1} that $\NL(\phi)$ implies
that $(\sE, \sF)$ is conservative. Thus any
equivalent statement of Theorem \ref{T:main} implies that the  process $X$ is conservative.

\begin{thm} \label{T:1.13}
Assume that the metric measure space $(M, d, \mu)$ satisfies $\VD$
and $\RVD$, and that
the functions $\phi_c$ and $\phi_j$ satisfy \eqref{polycon} and \eqref{e:1.9}.
 Let $\phi (r)$ be defined by \eqref{e:1.13}.
Then the following are equivalent: \smallskip \\
$(1)$ $\UHK(\phi_j,\phi_c)$ and
$(\sE, \sF)$ is conservative. \smallskip \\
$(2)$ $\UHKD(\phi)$, $\J_{\phi_j,\le}$ and $\E_\phi$. \smallskip \\
$(3)$ $\FK(\phi)$, $\J_{\phi_j,\le}$ and $\E_\phi$.\smallskip\\
$(4)$ $\FK(\phi)$, $\J_{\phi_j,\le}$ and $\Gcap (\phi)$.\smallskip \\
$(5)$ $\FK(\phi)$, $\J_{\phi_j,\le}$ and $\CSJ(\phi)$.
\end{thm}

We emphasize that the above two  theorems are equivalent characterizations and stability results. It is possible that none of the statements hold with a bad selection of $\phi_c$.

We point out that $\UHK(\phi_j,\phi_c)$ alone
does not imply the conservativeness of the associated Dirichlet form $(\sE, \sF)$.
See \cite[Proposition 3.1 and Remark 3.2]{CKW1} for more details. Under $\VD$, $\RVD$ and \eqref{polycon}, $\NDL(\phi)$ implies $\E_\phi$ (e.g.,\ see \cite[Proposition 3.5(2)]{CKW2}), and so $(3)$ in Theorem \ref{T:main} is stronger than $(2)$ in Theorem \ref{T:1.13}.
We also note that $\RVD$ is only used
in the implications that $\UHKD(\phi)\Longrightarrow \FK(\phi)$ and
 $\PI(\phi)\Longrightarrow \FK(\phi)$; see Proposition \ref{P:2.1}.
In particular, $(5)\Longrightarrow (1)$ in Theorem
 \ref{T:1.13} holds true under $\VD$ and \eqref{polycon}.

\medskip

Let $Z:=\{V_s,X_s\}_{s\ge0}$ be the space-time process corresponding to $X$, where $V_s=V_0-s$. The law of the space-time process $Z$ starting from $(t,x)$ will be denoted by $\bP^{(t,x)}$. For every open subset $D$ of $[0,\infty)\times M$, define
$\tau_D=\inf\{s>0:Z_s\notin D\}.$

\begin{definition} \rm \begin{description}
\item{(i)} We say that a Borel measurable function $u(t,x)$ on
$[0,\infty)\times M$ is \emph{parabolic} (or \emph{caloric}) on
$D=(a,b)\times B(x_0,r)$ for the process $X$ if there is a properly
exceptional set $\mathcal{N}_u$ associated with the process $X$ so that for every
relatively compact open subset $U$ of $D$,
$u(t,x)=\bE^{(t,x)}u(Z_{\tau_{U}})$ for every $(t,x)\in
U\cap([0,\infty)\times (M\backslash \mathcal{N}_u)).$

\item{(ii)}
We say that the \emph{parabolic Harnack inequality} ($\PHI(\phi)$) holds for the process $X$, if there exist constants $0<C_1<C_2<C_3<C_4$, $C_5>1$ and $C_6>0$ such that for every $x_0 \in M $, $t_0\ge 0$, $R>0$ and for
every non-negative function $u=u(t,x)$ on $[0,\infty)\times M$ that is parabolic on cylinder $Q(t_0, x_0,\phi(C_4R),C_5R):=(t_0, t_0+\phi(C_4R))\times B(x_0,C_5R)$,
\be\label{e:phidef}
 \esssup_{Q_- }u\le C_6 \,\essinf_{Q_+}u,
 \ee where $Q_-:=(t_0+\phi(C_1R),t_0+\phi(C_2R))\times B(x_0,R)$ and $Q_+:=(t_0+\phi(C_3R), t_0+\phi(C_4R))\times B(x_0,R)$.
\end{description}
\end{definition}

The following definition was introduced in \cite{BBK2} in the setting of graphs, and then extended in \cite{CKK2} to
the general setting of metric measure spaces.

\begin{definition}\label{thm:defUJS} {\rm
We say that $\UJS$ holds if
there is a symmetric function $J(x, y)$ so that
\eqref{e:1.15} holds, and
there is a constant $c>0$ such that for any $x,y\in M$ and $0<r\le d(x,y)/2$,
$$
J(x,y)\le \frac{c}{V(x,r)}\int_{B(x,r)}J(z,y)\,\mu(dz).
$$
} \end{definition}

The result below gives us the stable characterization of parabolic Harnack inequalities, as well as the relation between parabolic Harnack inequalities and two-sided heat kernel estimates.

\begin{theorem}\label{T:PHI}
 Suppose that the metric measure space $(M, d, \mu)$ satisfies $\VD$ and $\RVD$,
and that
 the functions $\phi_c$ and $\phi_j$ satisfy \eqref{polycon} and \eqref{e:1.9}.
  Let $\phi (r)$ be defined by \eqref{e:1.13}.
 Then
$$\PHI(\phi)\Longleftrightarrow \PI(\phi)+\J_{\phi,\le }+\CSJ(\phi)+\UJS.$$
Consequently,
$$\HK_-(\phi_j,\phi_c)\Longleftrightarrow \PHI(\phi)+\J_{\phi_j}.$$
 If, additionally, the metric measure space $(M,d,\mu)$ is connected and satisfies the chain condition, then
$$\HK(\phi_j,\phi_c)\Longleftrightarrow \PHI(\phi)+\J_{\phi_j}.$$
 \end{theorem}

Just like in \cite{CKW2, CKW4}, we could get more equivalent statements for $\PHI (\phi)$ in Theorem \ref{T:PHI}. We are not doing it here for the sake of space consideration.

\medskip

 We emphasize again that the connectedness and the chain condition of the underlying metric measure space $(M,d,\mu)$ is only used to derive optimal lower bounds of off-diagonal estimates for
the heat kernel
(i.e., from $\HK_-(\phi_j,\phi_c)$ to $\HK(\phi_j,\phi_c)$).
For  other statements in the three main results above,
the metric measure space $(M, d,\mu)$ is only assumed to satisfy the general VD and RVD; that is, neither do we assume
$M$ to be connected nor $(M, d)$ to be geodesic.
Furthermore, we do not assume the uniform comparability of volume of balls; that is, we do not assume
the existence of a non-decreasing function $V$ on $[0, \infty)$ with $V(0)=0$ so that
 $\mu(B(x,r))\asymp V(r)$ for all $x\in M$ and $r>0$.

\subsection{Remarks and discussions}\label{S:1.4}

In this subsection, we comment on the formulations of $\HK_- (\phi_j, \phi_c)$ and
$\HK (\phi_j, \phi_c)$, and discuss the relations of the main results to those in the literature.
Recall that the function $\phi$ is defined by \eqref{e:1.13}.

\begin{remark}\label{R:1.16} \rm

 \begin{itemize}
\item[(i)] The
scale function $\phi_c$ plays a role in the definitions of $\HK (\phi_j, \phi_c)$, $\HK_- (\phi_j,$
$\phi_c)$, $\UHK (\phi_j, \phi_c)$, $\UHKD (\phi)$, $\NL(\phi)$ and $\NDL(\phi)$ only
 for $t\le 1$ when $\beta_*>1$ and
for $t\ge1$ when $\beta^*>1$.
The cut-off time 1 here is not important. It can be replaced by any fixed constant $T>0$.
Furthermore, it follows from Theorem \ref{T:PHI} that $\HK_-(\phi_j, \phi_c)$
(and so $\HK(\phi_j, \phi_c)$) is stronger than $\PHI(\phi)$, which in turn yields the H\"{o}lder regularity of parabolic functions;
see the proof of \cite[Theorem 1.17]{CKW4}. In particular, this implies that
if $\HK_-(\phi_j, \phi_c)$ (respectively, $\HK (\phi_j, \phi_c)$) holds, then it can be strengthened to hold  for all $x,y \in M$,
and consequently the Hunt process $X$ can be refined to start from every point in $M$.

\item[(ii)] Note that when $d(x,y)\le \phi_c^{-1}(t)$,
\begin{equation}\label{e:1.25}
p^{(c)}(t,x,y)\simeq
 \frac{1}{V(x,\phi_c^{-1}(t))};
\end{equation}
and that
when $d(x,y)\ge \phi_c^{-1}(t)$,
for every $k\geq 1$,
\begin{equation}\label{e:1.27}\begin{split}
p^{(c)}(t,x,y)&\le    \frac{ c_1}{V(x, d (x, y))}\frac{ V(x, d(x, y))}{V(x,\phi_c^{-1}(t))}  \exp\left(-c_2\left(\frac{\phi_c(d(x,y))}{t}\right)^{1/(\beta_{2,\phi_c}-1)}\right)
  \\
&\leq     \frac{ c_3}{V(x, d (x, y))} \!\! \left( \frac{\phi_c (d(x, y))} {t} \right)^{d_2 /\beta_{1,\phi_c}} \!\!\!\!
 \exp\left(\!\!-c_2\left(\frac{\phi_c(d(x,y))}{t}\right)^{1/(\beta_{2,\phi_c}-1)}\right)
  \\
 &\le    \frac{c_4 (k)}{V(x, d (x, y))}\frac{t^k}{ \phi_c (d(x,y))^k},
 \end{split}\end{equation}
where in the first inequality we used \eqref{e:1.10} and
\eqref{e:effdiff},  in the second inequality we used VD and \eqref{polycon}, and the last inequality is a consequence
of the elementary inequality
$$
r^{k+d_2/ \beta_{1, \phi_c} } \exp\left(-c_2 r^{1/(\beta_{2,\phi_c}-1)}\right)\le c_5(k),\quad r\ge 1 .
$$
So in particular, we have
by \eqref{e:1.25} and \eqref{e:1.27} with $k=1$,
\begin{equation}
\label{e:1.26**}
p^{(c)}(t,x,y)\preceq \frac{1} {V(x,\phi_c^{-1}(t))} \wedge  \frac{ t}{V(x, d (x, y)) \phi_c (d(x,y))}
\quad \hbox{on } (0, \infty)\times M_0 \times M_0.
\end{equation}
Hence when $\beta^*>1$ and $\phi_c = \phi_j$ on $[1, \infty)$,
we have
$$
p^{(c)}(t,x,y)\preceq \frac{1} {V(x,\phi_j^{-1}(t))} \wedge  \frac{ t}{V(x, d (x, y)) \phi_j (d(x,y))}
\quad \hbox{on } [1, \infty)\times M_0 \times M_0.
$$
Consequently, in this case,
$$
\frac 1{V(x,\phi_c^{-1}(t))} \wedge\big(p^{(c)}(t,x,y)+p^{(j)}(t,x,y)\big)
\asymp p^{(j)}(t,x,y) \quad \hbox{on } [1, \infty)\times M_0 \times M_0.
$$
On the other hand, note that by \eqref{polycon}, there is some $k\geq 1$ so that $\phi_c (r)^k \geq c_6 \phi_j(r)$
on $[1, \infty)$. Thus it follows from \eqref{e:1.27}
that  for $t\in (0, 1]$ and $d(x, y) \geq 1$,
$$
p^{(c)}(t, x, y)
\le  \frac{c_7 t}{V(x, d (x, y))  \phi_j (d(x,y))}.
$$
This together with \eqref{e:1.25} and \eqref{e:1.26**} for $d(x, y) \leq 1$
shows that in the case of $\beta_*>1$ and $\phi_c=\phi_j$ on $[0, 1]$,
$$
\frac 1{V(x,\phi_c^{-1}(t))} \wedge\big(p^{(c)}(t,x,y)+p^{(j)}(t,x,y)\big)
\asymp p^{(j)}(t,x,y) \quad \hbox{on } (0, 1]\times M_0 \times M_0.
$$

Hence when $\phi =\phi_j $ on $[0, \infty)$ (that is, when $\phi_c = \phi_j$ on $[0, 1]$ if $\beta_*>1$
and $\phi_c=\phi_j$ on $(1, \infty)$ if $\beta^*>1$),  $\HK (\phi_j, \phi_c)$
is just the heat kernel estimate $p^{(j)}(t, x, y)$.
Thus in the case of $\phi=\phi_j$ on $[0, \infty)$,
Theorems \ref{T:main} and \ref{T:1.13}
 are essentially the main results (Theorems 1.13 and 1.15) of \cite{CKW1},
and Theorem \ref{T:PHI} is an essential part of \cite[Theorem 1.20]{CKW2}.
We note that by the proof of \cite[Lemma 4.1]{CKW1},
$\J_{\phi_j, \geq}$
implies $\PI (\phi_j)$
so in the case of $\phi=\phi_j$ on $[0,\infty)$, we can drop condition $\PI (\phi)$ from the statement of
Theorem \ref{T:main}.

\item[(iii)] When $\phi (r) =\phi_c (r)$ on $[0, \infty)$ (that is, when $\beta_*\wedge \beta^*>1$,
and  $\phi_c(r)$ is a strictly increasing function satisfying \eqref{polycon} and \eqref{e:1.9}),
$\HK_- (\phi_j,\phi_c)$ and
$\HK(\phi_j,\phi_c)$ are just HK$(\Phi,\psi)$  and  SHK$(\Phi,\psi)$ in \cite[Definition 2.8]{BKKL2}
with
$$
\psi=\phi_j \quad \hbox{and} \quad  \Phi=\phi_c.
$$
In this case, our Theorems \ref{T:main} and \ref{T:1.13} have also been independently obtained
in  \cite[Theorem 2.14, Corollary 2.15, Theorem 2.17 and Corollary 2.18]{BKKL2}.
See Remark \ref{R:1.18} below for more information.

\item[(iv)]
When $\phi (r)= \phi_c (r) {\bf1}_{[0, 1]} + \phi_j (r) {\bf1}_{(1, \infty)}$ (that is,
when $\beta_*>1$ and either $\beta^* \leq 1$ or $\beta^*>1$ with $\phi_c(r)=\phi_j(r)$ for all $r\in [1,\infty)$),
by \eqref{e:1.25} and \eqref{e:1.26**}, $\HK(\phi_j,\phi_c)$ is reduced into
\begin{equation} \label{e:1.31}
p(t,x,y)\asymp
\begin{cases}
 \frac 1{V(x,\phi_c^{-1}(t))} \wedge\big(p^{(c)}(t,x,y)+p^{(j)}(t,x,y)\big)\Big), \quad &0<t\le 1, \\
 p^{(j)}(t,x,y), &t> 1.
\end{cases}
\end{equation}
In
this case $\HK(\phi_j,\phi_c)$ is of the same form as that of
$\HK (\phi_c, \phi_j)$ in \cite{CKW4} for symmetric diffusions with jumps,
or equivalent, for symmetric Dirichlet forms that contain both the strongly local part and the pure jump part; see \cite[Definition 1.11 and Remark 1.12]{CKW4} for details. However, they are
quite different.
In \cite{CKW4} the function $\phi_c$ in $\HK (\phi_c, \phi_j)$ is
the scaling function of the diffusion (i.e. the strongly local part of Dirichlet forms), while in the present paper the function $\phi_c$ in $\HK (\phi_j, \phi_c)$ is
determined by $\phi_j$ and the underlying metric measure space $(M, d,\mu)$.

  \item[(v)] We note that the expression of $\HK(\phi_j,\phi_c)$ takes different form
	depending on whether $\beta_*\leq 1 $ (respectively, $\beta^*\leq 1$) or not.
	This is because when $\beta_*>1$ (respectively, $\beta^*>1$), the heat kernel estimate
	may involve another function $\phi_c$ that is intrinsically determined by $\phi_j$ but we do not have a generic formula for it under our general setting.
	However it does not necessarily mean that the heat kernel estimates for $p(t, x, y)$
	has a phase transition exactly at $\beta_*=1$ or $\beta^*=1$.
For example, suppose $(M,d,\mu)$ is
a connected metric measure space satisfying the chain condition on which there is a symmetric diffusion process enjoying the heat kernel estimate \eqref{Up} with $\beta>1$ as in Example \ref{Exam-M}. Consider a pure jump symmetric process $Y$ on this space whose jumping density has bounds
\eqref{eq:jownex1} with $\alpha_1<\beta$
and $\alpha_2>0$. Clearly, $\beta_*=\alpha_1$ and $\beta^*=\alpha_2$.
When $\alpha_2<\beta$,  the transition density function $p(t, x, y)$ of $Y$
has estimates \eqref{e:1.26} with
$\phi_j(r)=r^{\alpha_1} {\bf1}_{\{0<r\le 1\}}+ r^{\alpha_2}{\bf1}_{\{r>1\}}$
as in \cite{CKW1}, whereas for $\alpha_2>\beta$, $p(t, x, y)$ satisfies \eqref{e:1.5} as in Example \ref{Exam-M}. Hence
for this example, phase transition for the expression of heat kernel estimates for $p(t, x, y)$ occurs at $\alpha_2 = \beta$.

\end{itemize}
\end{remark}

In (ii)-(iv) of Remark \ref{R:1.16}, we have discussed the form of $\HK (\phi_j, \phi_c)$
for the cases of $\phi(r)=\phi_j (r)$ on $[0, \infty)$, $\phi(r)=\phi_c (r)$ on $[0, \infty)$,
and $\phi (r) = \phi_c (r) {\bf1}_{[0, 1]} + \phi_j (r) {\bf1}_{(1, \infty)}$. We now discuss
the remaining case of $\phi (r)$.

\begin{remark}\label{R:1.22} \rm This remark is concerned with
the case that
$\phi (r) = \phi_j (r) {\bf1}_{[0, 1]} + \phi_c (r) {\bf1}_{(1, \infty)}$.
It consists of
two subcases of $\beta^*>1$ where either $\beta_*\le 1$ or $\beta_*> 1$ with $\phi_c(r)=\phi_j(r)$ for all $r\in (0,1]$.
\begin{itemize}
\item[(i)]
In this case, we can
rewrite
the expression of $\HK(\phi_j,\phi_c)$ in the following way. For
$0<t\le 1$,
\begin{align*}
p(t,x,y)\simeq p^{(j)}(t,x,y)\simeq
\begin{cases}\frac{1}{V(x,\phi_j^{-1}(t))},\quad &d(x,y)\le c_1\phi_j^{-1}(t),\\
\frac{t}{V(x,d(x,y))\phi_j(d(x,y))},\quad&d(x,y)\ge
c_1\phi_j^{-1}(t).
\end{cases}
\end{align*}
For $t\ge 1$,
\begin{align*}
p(t, x, y) &\asymp
\frac 1{V(x,\phi_c^{-1}(t))} \wedge\big(p^{(c)}(  t,x,y)+p^{(j)}(t,x,y)\big) \\
&\asymp
\begin{cases}\frac{1}{V(x,\phi_c^{-1}(t))},\quad &d(x,y)\le c_2\phi_c^{-1}(t),\\
\frac{t}{V(x,d(x,y))\phi_j(d(x,y))}+\frac{1}{V(x,\phi_c^{-1}(t))}
\exp\left(- \frac{ d(x,y)}{\bar \phi_c^{-1}(t/d(x,y))}\right) ,
\quad&d(x,y)\ge c_2\phi_c^{-1}(t).
\end{cases}
 \end{align*}
See  the argument in Remark \ref{R:1.16}(ii).
In particular, for $t\in(0,1]$,
the heat kernel estimates $\HK (\phi_c, \phi_j)$
are completely
dominated by the jumping kernel for the associated Dirichlet form
$(\sE, \sF)$.
For $t\geq 1$, we have the following more
explicit
 expression of $\HK(\phi_j,\phi_c)$:
\begin{align*}p(t,x,y)\asymp\begin{cases}\frac{1}{V(x,\phi_c^{-1}(t))},\quad &d(x,y)\le c_2\phi_c^{-1}(t),\\
\frac{1}{V(x,\phi_c^{-1}(t))} \exp\left( - \frac{ d(x,y)} {\bar
\phi_c^{-1}(t/d(x,y))}\right),\quad& c_2\phi_c^{-1}(t)\le d(x,y)\le
t_*,\\
\frac{t}{V(x,d(x,y))\phi_j(d(x,y))},\quad &d(x,y)\ge t_*,
\end{cases}\end{align*}
where $t_*$ satisfies that
$$c_3\phi_c^{-1}(t)\log^{(\beta_{1,\phi_c}-1)/\beta_{2,\phi_c}}(\phi^{-1}_c(t)/\phi^{-1}_j(t))\le
t_*\le c_4
\phi_c^{-1}(t)\log^{(\beta_{2,\phi_c}-1)/\beta_{1,\phi_c}}(\phi^{-1}_c(t)/\phi^{-1}_j(t)),
$$
 and $\beta_{1,\phi_c}$ and $\beta_{2,\phi_c}$ are given in
\eqref{polycon}.
See the proof of Proposition \ref{thm:ujeuhkds} for
more details.
\begin{figure}[t]
\centerline{\epsfig{file=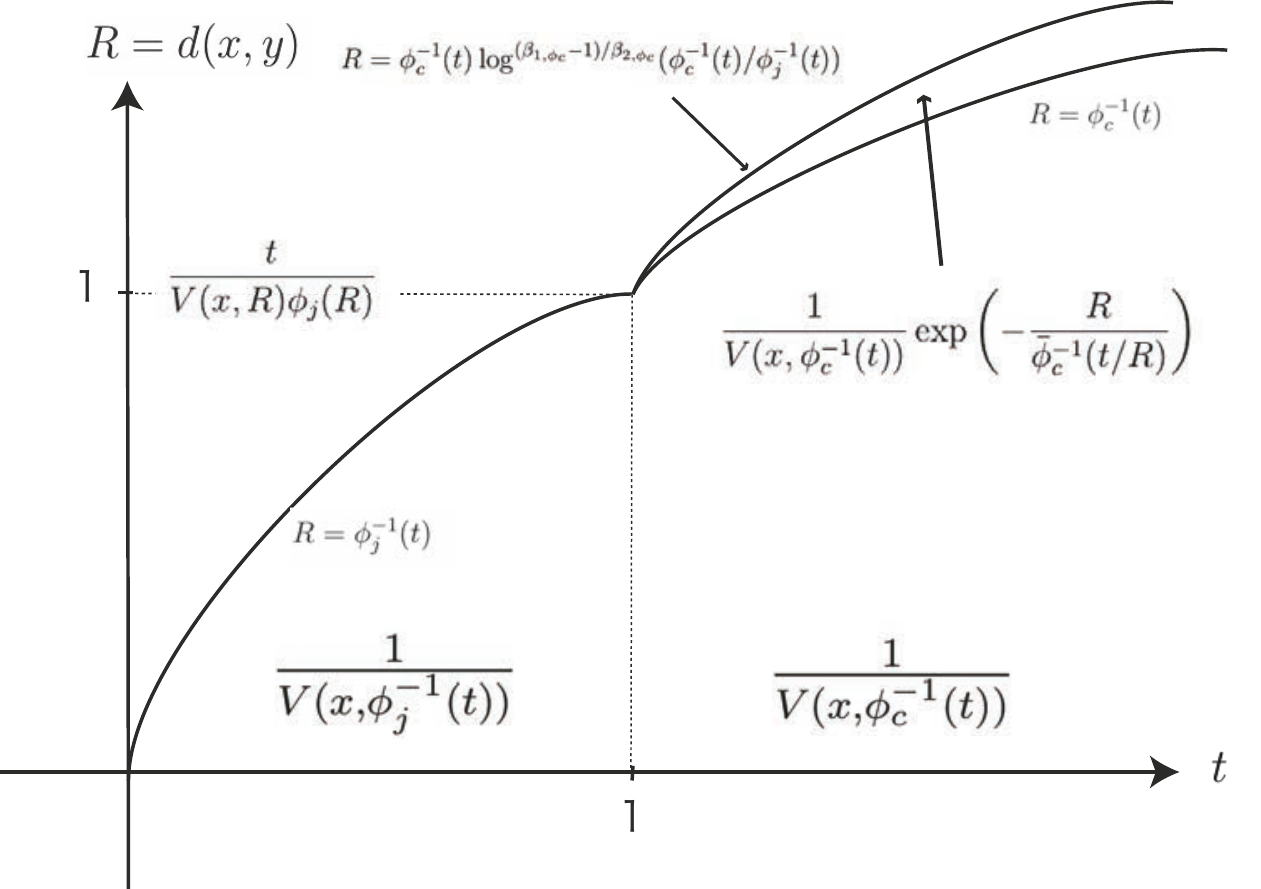, height=2in}}
\caption{Dominant term in the heat kernel estimates
$\HK (\phi_j, \phi_c)$ for $p(t,x,y)$ for the case in Remark \ref{R:1.22}.}\label{domfig}
\end{figure}
In this case, one can check that only the information of $\phi_c$ on $[1, \infty)$
is actually needed  for the expression of $\HK (\phi_j, \phi_c)$
because $t/d(x, y)\geq c_3>0$ holds
when $t\geq 1$ and $ c_2\phi_c^{-1}(t)\le d(x,y)\le t_*$.
Figure \ref{domfig}  indicates
in this case which term is the dominant one for the estimate of $p(t,x,y)$ in each region.

\item [(ii)]
The definition of $\HK (\phi_j, \phi_c)$ in this case is
 different from $\HK (\phi_c, \phi_j)$ in \cite{CKW4} for symmetric diffusions with jumps.
Denote by
 the heat kernel by $\bar p(t, x, y)$ in the case of symmetric diffusion with jumps. We say $\HK (\phi_c, \phi_j)$
holds if \eqref{e:1.31} holds.
Thus the expressions for $\HK (\phi_j, \phi_c)$ for symmetric jump processes with lighter jumping tails and $\HK (\phi_c, \phi_j)$
for symmetric diffusions with jumps are exactly switched over the time interval $(0, 1]$ and $(1, \infty)$.
The reason for this difference is as follows. For
 diffusions with jumps, due to the heavy tail property of jump phenomenon, the behaviors of the processes for large time are dominated by the pure jump part, and the behaviors for small time enjoy the continuous nature from diffusions as well as some interactions with jumps. For symmetric pure jumps processes with lighter tails
considered
in this case, the small time
 behavior of the processes are controlled by the jumping kernels; however, since large jumps of the associated process is so light (as a typical example, one can consider symmetric jump processes on $\bR^d$ whose associated jumping measure has finite second moments), it essentially also yields the diffusive property of the process when time becomes large.

\item[(iii)] The lower bound in $\HK_-(\phi_j,\phi_c)$ can be expressed more explicitly;
namely,
there are constants $c_1, c_2>0$ so that $$
p(t, x, y) \geq
c_1\begin{cases} p^{(j)}(t,x,y),&\quad 0<t\le 1,\\
\frac 1{V(x,\phi_c^{-1}(t))},&\quad t\geq 1, \, d(x,y)\le c_2\phi_c^{-1}(t),\\
\frac{t}{V(x,d(x,y))\phi_j(d(x,y))} ,&\quad t\geq 1, \, d(x,y)\ge c_2\phi_c^{-1}(t).
\end{cases}$$
\end{itemize}
\end{remark}

\begin{remark}\rm \label{R:1.18}
We make some remarks on the relations of our main results to those in \cite{BKKL2}.
\begin{itemize}
\item[(i)]
At the time we were finalizing this paper, we noticed an arXiv
preprint \cite{BKKL2} on a similar topic as  (part of)
this paper; see Remark \ref{R:1.16}(iii). The paper
\cite{BKKL2} is a partial  extension of \cite{CKW1} as well as \cite{BKKL1}.
Paper \cite{BKKL1} studies heat kernel estimates for symmetric jump processes on
$\bR^d$ with jumping kernels of mixed polynomial growth rate not necessarily
less than 2. It is strongly motivated
 by \cite{Mi}, which studied heat kernel estimates for subordinate Brownian motions
with scaling order not necessarily strictly less than 2,
 and uses some techniques from
\cite{CKW1}.  The heat kernel estimate in \cite{BKKL1} takes the form of
\begin{equation}\label{e:1.32}
 \frac1{V(x, \phi_c^{-1}(t))} \wedge \left( p^{(c)}(t, x, y)+ p^{(j)}(t, x, y)\right)
\end{equation}
with $V(x, r)=r^d$ and $\phi_c(r):=r^2/\int_0^r s/\phi_j(s) \,ds$.
Paper \cite{BKKL2} extends above formation of heat kernel estimates to metric measure spaces
with more general scaling function $\phi_c$ and studies its stable characterizations in the sprit of \cite{CKW1}.
However the formulation of \eqref{e:1.32} excludes some interesting
 cases such as those considered in
 Example \ref{Exam-M}, as it essentially requires the lower scaling index of $\phi_j$ on
$[0, 1]$ to be strictly larger than 1. It does not cover some cases studied in
\cite{CKW1} either when the global lower scaling index of $\phi_j$ is no larger than 1.
See part (ii) below for more information.

\quad The advantage of our setting and our
approach is that we start with two scaling functions $\phi_j$ and
$\phi_c$. This starting point is motivated by Example \ref{Exam-M},
and is quite natural. This allows us to treat the cases of $0<r\le 1$ and $r>1$
separably with possibly different scaling indices.
It also allows us to incorporate the heat kernel formulation \eqref{e:1.32}
considered in \cite{BKKL1, BKKL2}.
The significance
of this viewpoint is further illustrated by Example
\ref{e:example5.2}, where the local lower scaling index of the scale
function $\phi_j$ on $(1,\infty)$ is strictly larger than 1, while
the local lower scaling index of
 the scale function $\phi_j$ on $(0,1)$ can take any value in
 $(0,2)$. Both Example \ref{Exam-M} and Example \ref{e:example5.2} fall outside
the main settings of \cite{BKKL1, BKKL2}.
However, see a less precise heat kernel estimates in \cite[Theorem 2.19]{BKKL2}.

\quad A key step in our approach is Proposition \ref{thm:ujeuhkds}, whose
proof is the most difficult part of our paper as mentioned above.
The proof of Proposition \ref{thm:ujeuhkds} is strongly motivated by
that of \cite[Proposition 4.6]{CKW4}
for symmetric diffusion with
jumps. (Indeed, the main results of this paper indicates the similar
philosophy of heat kernel estimates between diffusions with jumps in
\cite{CKW4} and pure jump Markov processes
having lighter tails considered in this paper
under the viewpoint that time interval $t\le 1$ and $t>1$ switched.)
The counterpart of our Proposition \ref{thm:ujeuhkds} in
\cite{BKKL2} is Theorem 2.13 there.
Actually, \cite[Theorem 2.13]{BKKL2} plays a crucial role in all the main results in \cite{BKKL2}. With this at hand, the authors fully used approaches from \cite{CKW1, CKW2} to establish the stability characterizations of heat kernel estimates \cite{BKKL2}; see Theorem 2.14 and Corollary 2.15; Theorem 2.17 and Corollary 2.18  in that paper for more details.
The proof of \cite[Theorem 2.13]{BKKL2} is based on that of \cite[Theorem 4.5]{BKKL1}, which
makes full use of the property that the global lower weak scaling
index of $\phi_c$ (or $\Phi$ in their notation) is strictly larger than 1. Though this
self-improvement approach looks a little similar, the proof of
Proposition \ref{thm:ujeuhkds} of this paper is much more delicate as our
setting is more general.
Our approach works for
the setting of \cite{BKKL1}
and beyond.
Moreover, our heat kernel estimate $\HK (\phi_j, \phi_c)$ is more
concise and explicit, see Remark \ref{R:1.22}(i). Furthermore, our
approach also yields the stable characterizations of parabolic
Harnack inequalities, which are useful as indicated in Example
\ref{Exm-phi}. Parabolic Harnack
inequalities are not studied in \cite{BKKL2}.

\item[(ii)]
It is assumed in \cite{BKKL2} that
$\Phi (r) \leq c_1 \psi (r)$
for all $r\geq 0$ and that the global lower scaling index of
$\Phi$
is strictly larger than 1. It follows that
$\beta^*>1$
for $\phi_j:=\psi$. Although it is not mentioned explicitly in the corresponding results in \cite{BKKL2},
for HK$(\Phi,\psi)$ and  SHK$(\Phi,\psi)$ to hold, the lower scaling index $\beta_*$
has to be strictly larger than 1 as well,  if $\inf_{x\in M} V(x, 1)>0$ and there are some constants $c_1>0$ and
$\beta_2 \in (0, 2)$
so that
\begin{equation}
 \phi_j (R)/\phi_j (r) \leq c_1 (R/r)^{\beta_2}
\quad \hbox{for } 0<r<R\leq 1.
\end{equation}
Below is the reasoning.
Using the notation of this paper (that is, $\psi=\phi_j$ and
$\Phi=\phi_c$),
define
$$
\bar \phi_j (r)=\phi_j  (r){\bf1}_{[0, 1]}(r) + r^{\beta^*\wedge (3/2)}{\bf1}_{[1, \infty)}(r).
$$
Clearly there is a constant $c_2>0$ so that  $\bar \phi_j (r) \leq c_2 \phi_j(r)$ on $[0, \infty)$
and
$$
J(x, y) \leq \frac{c_2}{V(x, d(x, y)) \bar \phi_j (d(x, y))}
\quad \hbox{for  } d(x, y)\geq  1.
$$
Let
$$
\bar J(x, y):= J(x, y){\bf1}_{\{d(x, y)\leq 1\}} + \frac{c_2}{V(x, d(x, y))
\bar \phi_j (d(x, y))} {\bf1}_{\{d(x, y)> 1\}}.
$$
Clearly, $J(x, y) \leq \bar J(x, y)$ and $\| \bar J-J\|_\infty \leq {c_3}/{\inf_{x\in M}
V(x, 1)} <\infty$. By the calculation in \cite[Lemma 2.1]{CKW1}, one can show that
$$  {\cal J}(x):=\int_M \left( \bar J(x, y) -J(x, y))\right) \,\mu(dy)
\leq  \int_M  \frac{c_2}{V(x, d(x, y))
\bar \phi_j (d(x, y))} {\bf1}_{\{d(x, y)>1\}}\, \mu(dx) \leq c_4
$$
for every $x\in M$. Denote by $\bar p(t, x, y)$ the transition density function for the symmetric Hunt process $\bar X$ with jumping kernel $\bar J(x, y)$. It follows from
\cite[Lemma 3.1]{BGK2} and \cite[Lemma 3.6]{BBCK} that
$$
     e^{-c_4 t} p(t, x, y) \leq \bar p(t, x, y) \leq p(t, x, y) +t \| \bar J -J\|_\infty
		\quad \hbox{for } t>0 \hbox{ and } x, y\in M_0.
$$
Since $p(t, x, y)$ satisfies HK$(\Phi, \psi)$, which is $\HK_- (\phi_j, \phi_c)$
in the nation of this paper, we have
\begin{equation}\label{e:1.33}
\bar p(t, x, y) \simeq \frac{1}{V(x, \phi_c^{-1} (t))}
\quad \hbox{for } t\in (0, 1] \hbox{ and } d(x, y) \leq \phi_c^{-1}(t).
\end{equation}
On the other hand, since
$$
\bar \phi_j(R)/\bar \phi_j (r)\leq c_5 (R/r)^{\bar \beta_2}
$$
for  every $R>r>0$, where $\bar \beta_2:= \beta_2 \vee (\beta^*\wedge (3/2))<2$,
we have by \cite[Remark 1.7 and Theorem 1.13]{CKW1},
$\bar p(t, x, y)$ enjoys the heat kernel
estimate  $\HK(\bar \phi_j)$. In particular,
$$
\bar p (t, x, y) \simeq \frac1{V(x, \phi_j^{-1} (t)) } \quad \hbox{for } t\in (0, 1]
\hbox{ and } d(x, y) \leq \phi_j^{-1}(t).
$$
It follows then from \eqref{e:1.33} that
$$
\frac1{V(x, \phi_j^{-1} (t)) }\simeq \frac1{V(x, \phi_c^{-1} (t)) }
\quad \hbox{for } t\in (0, 1].
$$
This won't be impossible under $\VD$ and $\RVD$ unless $\beta_*>1$, since $\phi_c $ is not comparable to $\phi_j$ on $(0, 1]$ when $\beta_*\leq 1$.
In summary, \cite{BKKL2} essentially requires
$\beta_*\wedge \beta^*>1$.
\end{itemize}
\end{remark}

As mentioned in Remark \ref{R:1.16}, the setting of this paper covers that of \cite{CKW1}
as well as
\cite{BKKL2}.
 For simplicity, for proofs, we will be mainly concerned with the case
that $\phi (r)=\phi_j (r){\bf1}_{[0, 1]} + \phi_c (r){\bf1}_{(1, \infty)}$; that is,
the case discussed in Remark \ref{R:1.22}. The other cases  can be treated similarly and even easier. On the other hand,
as mentioned in Remark \ref{R:1.22}(ii), the expressions of heat kernel estimates for the process studied in this case
are the same as these for diffusions with jumps studied in \cite{CKW4} but with time interval  $t\le 1$ and $t>1$ switched. So some ideas and strategies from \cite{CKW4} can be used in this
case. Indeed,
one can follow the arguments in \cite{CKW4} to prove Theorems \ref{T:main} and \ref{T:1.13} in this case, as well as other main assertions, by exchanging the corresponding time intervals. Hence we will skip most of proofs when they are the same, but will illustrate explicitly the differences in the proofs when they are different.
 We in particular give the proof of the $(2)\Longrightarrow (1)$ part in Theorem
 \ref{T:1.13}
 (see Proposition \ref{thm:ujeuhkds} below), which is the most difficult one among all the proofs.

\medskip

The rest of the paper is organized as follows.
 In Section \ref{S:2}, we consider the consequences of $\UHK(\phi_j,\phi_c)$, and prove the $(1) \Longrightarrow (4)$ part
 and the $(4) \Longrightarrow (5)$ part
  of Theorem \ref{T:1.13}.
 Section \ref{section4} is devoted to proving $\UHK(\phi_j,\phi_c)$, which is the most difficult part of the paper.
The crucial step is how to use
 rough tail probability estimates to get $\UHK(\phi_j,\phi_c)$. In particular, $(5) \Longrightarrow (3)$, $(3) \Longrightarrow (2)$  and $(2)\Longrightarrow (1)$ in
 Theorem \ref{T:1.13} are proven here. With all the results in previous sections, we prove two-sided heat kernel estimates (i.e., $\HK_-(\phi_j,\phi_c)$ and $\HK(\phi_j,\phi_c)$)
 in Section \ref{section6}. In Section \ref{section7}, we
 also give
 some applications of heat kernel estimates. In particular, the characterizations of parabolic Harnack inequalities are presented here.
 Finally, as mentioned earlier,
 the proof of Example \ref{Exam-M} is given in the appendix.

Throughout this paper, we will use $c$, with or without subscripts,
to denote strictly positive finite constants whose values are
insignificant and may change from line to line. For $p\in [1,
\infty]$, we will use $\| f\|_p$ to denote the $L^p$-norm in
$L^p(M;\mu)$. For $B=B(x_0, r)$ and $a>0$, we use $a B$ to denote
the ball $B(x_0, ar)$. For any subset $D$ of $M$, $D^c:=M\setminus
D$.

\section{Implications of heat kernel estimates} \label{S:2}

In this section, we will prove (1) $\Longrightarrow$ (3)
and (3) $\Longrightarrow$ (4)
in Theorem
\ref{T:1.13}.
We first recall the following fact, whose proof is the same as that
of \cite[Proposition 7.6]{CKW1}.

\begin{prop}\label{P:2.1}
Under $\VD$, $\RVD$ and \eqref{polycon},
we have $\UHKD(\phi)\Longrightarrow\FK(\phi)$ and $\PI(\phi) \Longrightarrow\ \FK(\phi)$.
\end{prop}

\subsection{$\UHK(\phi_j,\phi_c) +
(\sE, \sF) \hbox{ is conservative}
\Longrightarrow \J_{\phi_j,\le}$
and $\HK_-(\phi_j,\phi_c) \Longrightarrow \J_{\phi_j}$}

\begin{proposition}\label{l:jk}
Under $\VD$ and \eqref{polycon},
$$
\UHK(\phi_j,\phi_c)
\hbox{ and } (\sE, \sF) \hbox{ is conservative}
\Longrightarrow \J_{\phi_j,\le}
$$ and $$\HK_-(\phi_j,\phi_c) \Longrightarrow \J_{\phi_j}.$$
In particular, $$\HK(\phi_j,\phi_c) \Longrightarrow \J_{\phi_j}.$$\end{proposition}

\begin{proof}
The proof is similar to that of \cite[Proposition 3.3]{CKW1} or \cite[Proposition 3.1]{CKW4}, so
we omit it here. We remark that in fact we only need $\UHK(\phi_j,\phi_c)$ (resp. $\HK_-(\phi_j,\phi_c)$ or $\HK(\phi_j,\phi_c)$) of $p(t, x, y)$ for sufficiently small $t$.
\end{proof}

\subsection{$\UHK(\phi_j,\phi_c)
+(\sE, \sF) \hbox{ is conservative}
\Longrightarrow\Gcap(\phi)+ \CSJ(\phi)$}

In this subsection, we give the proof that $\UHK(\phi_j,\phi_c)$ and
the conservativeness of $(\sE, \sF)$ imply $\Gcap(\phi)$ and $\CSJ(\phi)$.

The following lemma is frequently used in the proofs of the paper.
 \begin{lemma}\label{intelem}$(${\rm\cite[Lemma 2.1]{CKW1}}$)$ Assume that $\VD$ and
 \eqref{polycon} hold. If $\J_{\phi_j,\le}$ holds,
then there exists a constant $c_1>0$ such that
\[
 \int_{B(x,r)^c}
J(x,y)\,\mu (dy)\le \frac{c_1}{\phi_j (r)}
\quad\mbox{for every }
x\in M \hbox{ and } r>0.
\] Similarly,
if $\J_{\phi,\le}$ holds,
then there exists a constant $c_2>0$ such that
\[
\int_{B(x,r)^c}
J(x,y)\,\mu (dy)\le \frac{c_2}{\phi (r)}
\quad\mbox{for every }
x\in M \hbox{ and } r>0.
\]
\end{lemma}

\begin{proposition}\label{Conserv} Assume that
$\VD$, \eqref{polycon} and
$\UHK(\phi_j,\phi_c)$ hold and that $(\sE, \sF)$ is conservative.
Then $\EP_{\phi,\le}$ holds.
Consequently, $\Gcap (\phi)$ holds.
\end{proposition}

\begin{proof}
According to the proof of \cite[Lemma 2.7]{CKW1} and the
conservativeness of the process $X$, we only need to verify that
there is a constant $c_1>0$ such that for each $t,r>0$ and for
almost all $x\in M$,
$$
\int_{B(x,r)^c} p(t, x,y)\,\mu(dy)\le \frac{c_1 t}{\phi(r)}.
$$
For this, we only need to consider the case that $\phi(r)>t$; otherwise, the inequality above holds trivially with $c_1=1$.

Recall the definition of $\phi (r)$ from \eqref{e:1.13}.
In the following, we only consider the case that
$\phi (r) =\phi_j (r) {\bf1}_{[0, 1]}(r) + \phi_c (r) {\bf1}_{(1, \infty)}$;
other cases can be treated similarly or easily. In this case, $\UHK(\phi_j,\phi_c)$
is that for any $x,y\in M_0$ and $t>0$,
$$p(t,x,y)\le c^*_1\begin{cases} p^{(j)}(t,x,y),&\quad 0<t\le 1,\\
\frac 1{V(x,\phi_c^{-1}(t))} \wedge\big(p^{(c)}( c^*_2
t,x,y)+p^{(j)}(t,x,y)\big)\Big),&\quad t\ge1,\end{cases}$$ where $c^*_1,c^*_2>0$ are independent of $x,y\in M_0$ and $t>0$.

\smallskip

  For $0<t\le 1$, it follows from the proof of \cite[Lemma 2.7]{CKW1} that the desired assertion holds.
 Hence we consider the case $t\ge1$.
 According to
  the above expression of $\UHK(\phi_j,\phi_c)$,
   $\VD$
	and \eqref{polycon}--\eqref{e:effdiff},
	for any $t\ge1$ and $r\ge1$ with
	 $ \phi_c(r)>t$ and every $x\in M_0$,
\begin{align*}
&\int_{B(x,r)^c}p(t, x,y)\,\mu(dy)\\
&= \sum_{i=0}^\infty\int_{B(x,2^{i+1}r)\setminus B(x,2^ir)}p(t, x,y)\,\mu(dy)\\
&\le  \sum_{i=0}^\infty \frac{c_2 V(x,2^{i+1}r)}{V(x,\phi_c^{-1}(t))}\exp\left(-\frac{c_3 2^i r}{\bar\phi_c^{-1}(t/(2^ir))}\right)+\sum_{i=0}^\infty \frac{c_2t V(x,2^{i+1}r)}{V(x,2^{i}r)\phi_j(2^ir)}\\
&\le  c_4\sum_{i=0}^\infty \left(\frac{2^ir}{\phi_c^{-1}(t)}\right)^{d_2}\exp\left(-\frac{c_5 2^i r}{\bar\phi_c^{-1}(t/r)}\right)+ \frac{c_4t}{\phi_j(r)}\sum_{i=0}^\infty 2^{-i\beta_{1,\phi_j}}\\
&\le c_6\sum_{i=0}^\infty\left(2^i\left(\frac{\phi_c(r)}{t}\right)^{1/\beta_{1,\phi_c}}\right)^{d_2}\exp\left(-c_72^i\left(\frac{\phi_c(r)}{t}\right)^{1/(\beta_{2,\phi_c}-1)}\right)+\frac{c_6t}{\phi_j(r)}\\
&\le c_6\sum_{i=0}^\infty\left(2^{i(\beta_{2,\phi_c}-1)}\frac{\phi_c(r)}{t}\right)^{d_2(1+1/(\beta_{2,\phi_c}-1))}\exp\left(-c_7\left(2^{i(\beta_{2,\phi_c}-1)}\frac{\phi_c(r)}{t}\right)^{1/(\beta_{2,\phi_c}-1)}\right)+\frac{c_6t}{\phi_j(r)}\\
&\le  c_8\sum_{i=0}^\infty\exp\left(-\frac{c_7}{2}\left(2^{i(\beta_{2,\phi_c}-1)}\frac{\phi_c(r)}{t}\right)^{1/(\beta_{2,\phi_c}-1)}\right)+\frac{c_6t}{\phi_j(r)}\\
&\le c_{9}t\left(\frac{1}{\phi_c(r)}+\frac{1}{\phi_j(r)}\right)\le \frac{c_{10}t}{\phi(r)},
\end{align*}
where in the second inequality we used the fact that $\bar \phi_c(r)$ is increasing on $(0,\infty)$, the fourth inequality we used the fact that $\beta_{1,\phi_c}>1$,  in the fifth and the sixth inequalities we used the facts that for there is a constant $c_{11}>0$ such that
$$s^{d_2(1+1/(\beta_{2,\phi_c}-1))}\le c_{11}\exp\left(\frac{c_7}{2}s^{1/(\beta_{2,\phi_c}-1)}\right),\quad s\ge 1$$ and
$$\sum_{i=0}^\infty \exp\Big(-\frac{c_7}{2}
(2^{i(\beta_{2,\phi_c}-1)}s)^{1/(\beta_{2,\phi_c}-1)}\Big)
\le c_{11}/s,\quad s\ge 1,$$ respectively,
and the last inequality is due to \eqref{e:1.9}.
 The proof of  $\EP_{\phi,\le, \eps}
\Longrightarrow \Gcap (\phi)$ is the same as that of \cite[Lemma 2.8]{GHH},
\qed
\end{proof}

\begin{proposition} \label{P:gdCS}
Under $\VD$ and \eqref{polycon}, $\Gcap(\phi)$ and $\J_{\phi,\le}$ imply $\CSJ(\phi)$. In particular,
suppose that $\VD$, \eqref{polycon} and $\UHK(\phi_j,\phi_c)$
hold, and that $(\sE, \sF)$ is conservative.
Then $\CSJ(\phi)$ holds.
\end{proposition}

\begin{proof}
The first assertion follows from that of \cite[Lemma 2.4]{GHH} (or that of \cite[Proposition 2.5]{CKW4}). This along with Proposition \ref{l:jk} and Lemma \ref{Conserv} gives us the second one. Indeed, we also can follow the argument of \cite[Proposition 3.6]{CKW1} to verify the second assertion directly.
 \qed\end{proof}

\section{Implications of $\FK(\phi)$, $\CSJ (\phi)$ and $\J_{\phi_j, \le}$ $(${\bf or} $\J_{\phi, \le})$}\label{section4}

\subsection{$\FK(\phi)+\J_{\phi,\le}+\CSJ(\phi)\Rightarrow \E_\phi$, and $\FK(\phi)+\J_{\phi,\le}+ \E_\phi\Rightarrow \UHKD(\phi)$}
\label{S:3.1}

The next two propositions can be proved by following arguments in \cite{CKW1}.

\begin{proposition} \label{P:exit}
Assume $\VD$, \eqref{polycon}, $\FK(\phi)$, $\J_{\phi,\le}$ and $\CSJ(\phi)$ hold. Then $\E_{\phi}$ holds.
\end{proposition}

\begin{proof} According to \cite[Lemma 4.14]{CKW1}, $\E_{\phi,\le} $ holds under $\VD$, \eqref{polycon} and $\FK(\phi)$.
On the other hand, following the arguments of \cite[Lemmas 4.15,
4.16 and 4.17]{CKW1}, we can obtain $\E_{\phi,\ge}$. Here, we note
that under assumptions of proposition, one can obtain mean value
inequalities in the present setting as those in \cite[Proposition
4.11 and Corollary 4.12]{CKW1}, and so the arguments above go
through. \qed
\end{proof}

\begin{proposition}\label{ehi-ukdd} Suppose that $\VD$, \eqref{polycon}, $\FK(\phi)$, $\E_\phi$ and $\J_{\phi, \le}$ hold. Then $\UHKD(\phi)$ is satisfied, i.e., there is a constant $c>0$ such that for all $x\in M_0$ and $t>0$,
$$p(t, x,x)\le \frac{c}{V(x,\phi^{-1}(t))}.$$ \end{proposition}
\begin{proof} The proof is the same as that of \cite[Theorem 4.25]{CKW1}, and so we omit it here.
\qed\end{proof}

 Propositions \ref{P:exit} and \ref{ehi-ukdd} give the
$(5)\Longrightarrow(3)$ part  and the
$(3)\Longrightarrow(2)$ part of Theorem \ref{T:1.13}.

\subsection{$\UHKD(\phi)+\J_{\phi_j,\le}+ \E_\phi\Rightarrow \UHK(\phi_j,\phi_c)$}
\label{S:3.2}

In this part, we will prove $(2)\Longrightarrow(1)$ in Theorem
\ref{T:1.13}. For this, we need to use the truncation technique. Fix
$\rho>0$ and define a bilinear form $(\sE^{(\rho)}, \sF)$ by
$$\sE^{(\rho)}(u,v)=\int (u(x)-u(y))(v(x)-v(y)){\bf 1}_{\{d(x,y)\le
\rho\}}\, J(dx,dy).$$ Clearly, the form $\sE^{(\rho)}(u,v)$ is well
defined for $u,v\in \sF$, and $\sE^{(\rho)}(u,u)\le \sE(u,u)$ for
all $u\in \sF$. Assume that $\VD$, \eqref{polycon} and
$\J_{\phi,\le}$ hold. Then we have by Lemma \ref{intelem}
 that for all $u\in \sF$,
\begin{align*}
\sE(u,u)-\sE^{(\rho)}(u,u)&= \int(u(x)-u(y))^2{\bf 1}_{\{d(x,y)>\rho\}}\,J(dx,dy)\\
&\le 4\int_Mu^2(x)\,\mu(dx)\int_{B(x,\rho)^c}J(x,y)\,\mu(dy)\le
\frac{c_0\|u\|_{2}^2 }{\phi(\rho)}.
\end{align*}
Thus, $\sE_1 (u, u)$ is equivalent to $\sE^{(\rho)}_1(u,u):=
\sE^{(\rho)}(u,u)+ \|u\|_2^2$ for every $u\in \sF$. Hence
$(\sE^{(\rho)}, \sF)$ is a regular Dirichlet form on $L^2(M; \mu)$.
Throughout this paper, we call $(\sE^{(\rho)}, \sF)$
the $\rho$-truncated
Dirichlet form. The Hunt process associated with
$(\sE^{(\rho)}, \sF)$, denoted by $(X_t^{(\rho)})_{t\ge0}$, can be
identified in distribution with the Hunt process of the original
Dirichlet form $(\sE, \sF)$ by removing those jumps of size larger
than
 $\rho$.

According to the proof of \cite[Lemma 5.2]{CKW1}, we have the
following statement.

\begin{lemma}\label{P:truncated}
Assume that $\VD$, \eqref{polycon},
$\UHKD(\phi)$, $\J_{\phi,\le}$ and $\E_{\phi}$ hold. Then the $\rho$-truncated Dirichlet form $(\sE^{(\rho)}, \sF)$ has the heat kernel $q^{(\rho)}(t,x,y)$, and
it holds that
for any $t> 0$ and all $x,y\in M_0$,
$$q^{(\rho)}(t,x,y)\le c_1\bigg(\frac{1}{V(x,\phi^{-1}(t))}+ \frac{1}{V(y,\phi^{-1}(t))}\bigg) \exp\left(c_2\frac{t}{\phi(\rho)}-c_3\frac{d(x,y)}{\rho}\right),$$ where $c_1,c_2,c_3$ are positive constants independent of $\rho$.

Consequently, the following holds
for any $t> 0$ and all $x,y\in M_0$,
$$q^{(\rho)}(t,x,y)\le \frac{c_4}{V(x,\phi^{-1}(t))}\bigg(1+ \frac{d(x,y)}{\phi^{-1}(t)}\bigg)^{d_2} \exp\left(c_2\frac{t}{\phi(\rho)}-c_3\frac{d(x,y)}{\rho}\right).$$
\end{lemma}

From here to the end of this section, we will assume $\J_{\phi_j,\le}$ instead of $\J_{\phi,\le}$. First, the relation between $p(t,x,y)$ and $q^{(\rho)}(t,x,y)$ can be seen through the Meyer's decomposition; e.g.\ see \cite[Section 7.2]{CKW1}. In particular, according to \cite[(4.34) and Proposition 4.24 (and its proof)]{CKW1}, for any $t,\rho>0$ and all $x,y\in M_0$,
\begin{equation}\label{e:trun-org} p(t,x,y)\le q^{(\rho)}(t,x,y)+ \frac{c_1t}{V(x,\rho)\phi_j(\rho)}\exp\left(\frac{c_1t}{\phi(\rho)}\right),\end{equation}
where $c_1$ is independent of $t,\rho>0$ and $x,y\in M_0.$
With this at hand, in order to deduce full upper bounds for
the heat kernel, we will make use of tail estimates of the transition
probability.
Namely,
we will use the following lemma, which is essentially taken from \cite[Lemma 4.3]{BKKL1} and partly motivated by the proof of \cite[Proposition 5.3]{CKW1}.

\begin{lemma}\label{L:selftail}
Assume that $\VD$, \eqref{polycon},
$\UHKD(\phi)$, $\J_{\phi_j,\le}$ and $\E_\phi$ hold. Let $f:\bR_+\times \bR_+\to \bR_+$ be a measurable function satisfying that $t\mapsto f(r,t)$ is non-increasing for all $r>0$, and that $r\mapsto f(r,t)$ is non-decreasing for all $t>0$. Suppose that the following hold:
\begin{itemize}
\item[{\rm (i)}] For each $b>0$, $\sup_{t>0} f(b\phi^{-1}(t),t)<\infty$.
\item [{\rm (ii)}] There exist constants $\eta\in (0,\beta_{1,\phi_j}]$ and $a_1,c_1>0$ such that for all $x\in M_0$ and $r,t>0$,
$$\int_{B(x,r)^c} p(t,x,y)\,\mu(dy)\le c_1\left(\frac{\phi_j^{-1}(t)}{r}\right)^\eta+c_1\exp\left(-a_1 f(r,t)\right).$$\end{itemize}
Then, there exist constants $k,c_0>0$ such that for all $x,y\in M_0$ and $t>0$,
\begin{align*}p(t,x,y)\le &\frac{c_0t}{V(x,d(x,y))\phi_j(d(x,y))}\\
&+\frac{c_0}{V(x,\phi^{-1}(t))}\left(1+\frac{d(x,y)}{\phi^{-1}(t)}\right)^{d_2}\exp\left(-a_1k f(d(x,y)/(16 k),t)\right).\end{align*}

Furthermore, the conclusion still holds true for any $t\in (0,T]$ or
$t\in [T,\infty)$ with some $T>0$, if assumptions {\rm(i)} and {\rm(ii)}
above are restricted on the corresponding time interval.
\end{lemma}
\begin{proof}
The proof is the same as that of
\cite[Lemma 4.5]{CKW4} so it is omitted here. \qed
\end{proof}
\begin{proposition}\label{thm:ujeuhkds}
Suppose $\VD$ and \eqref{polycon}, $\UHKD(\phi)$, $\J_{\phi_j,\le}$ and $\E_{\phi}$ hold.
Then we have $\UHK(\phi_j,\phi_c)$.
\end{proposition}

While the idea of the proof of
Proposition \ref{thm:ujeuhkds} is similar to that of
\cite[Proposition 4.6]{CKW4}, the actual proof requires lots of modifications.
To illustrate the differences of the proofs between \cite{CKW4} and the
present setting, below we give the proof of Proposition
\ref{thm:ujeuhkds} in details.
For the proof, we need
the following two lemmas in addition to Lemma \ref{L:selftail}.

\begin{lemma}\label{L:lemmastep1}
Suppose $\VD$ and \eqref{polycon}, $\UHKD(\phi)$, $\J_{\phi_j,\le}$ and $\E_{\phi}$ hold.
Then there exist constants $a, c>0$ and $N\in \bN$ such that for all $x,y\in M_0$ and $t>0$,
$$
p(t,x,y)\le \frac{ct}{V(x,d(x,y))\phi_j(d(x,y))}+ \frac{c}{V(x,\phi^{-1}(t))}\exp\left(-\frac{a d(x,y)^{1/N}}{\phi^{-1}(t)^{1/N}}\right).
$$
\end{lemma}

\begin{proof} The proof is the same as that of
\cite[Lemma 4.7]{CKW4}. \qed
\end{proof}

\begin{lemma}\label{L:keyexist}
For any $C_*>0$, there exist constants $C_0,C^*>0$ such that
\begin{itemize}
\item[{\rm(i)}] for any $t\in [1,\infty)$ and $r\ge C_0\phi_c^{-1}(t)$, the function
$$r\mapsto F_{1,t}(r):=\frac{\exp\left(C_*r/\bar\phi_c^{-1}(t/r)\right)}{r/\bar\phi_c^{-1}(t/r)}$$ is strictly increasing.
\item[{\rm(ii)}] for any $t\in [1,\infty)$, there is a unique $r_*(t)\in (C_0\phi_c^{-1}(t),\infty)$ such
that $ F_{1,t}(r_*(t))=F_{2,t}(r_*(t)),$
$$ F_{1,t}(r)<F_{2,t}(r),\quad r\in (C_0\phi_c^{-1}(t), r_*(t))$$
and $$F_{1,t}(r)>F_{2,t}(r),\quad r\in ( r_*(t),\infty),$$ where
$$F_{2,t}(r)=\frac{C^*\bar\phi_c^{-1}(t/r)}{\phi_j^{-1}(t)}.$$
\end{itemize}\end{lemma}
\begin{proof} (i) We know from \eqref{polycon} and \eqref{e:effdiff} that, if $r\ge C_0\phi_c^{-1}(t)$ with $C_0$ large enough, then
$$\frac{r}{\bar\phi_c^{-1}(t/r)}\ge c_1 \left(\frac{\phi_c(r)}{t}\right)^{1/(\beta_{2,\phi_c}-1)}\ge c_2\left(\frac{r}{\phi_c^{-1}(t)}\right)
^{\beta_{1,\phi_c}/(\beta_{2,\phi_c}-1)}\ge c_2
C_0^{\beta_{1,\phi_c}/(\beta_{2,\phi_c}-1)},$$ where the constants
$c_1$ and $c_2$ are independent of $C_0$. Note further that the
function $r\mapsto r/\bar\phi_c^{-1}(t/r)$ is strictly increasing
due to the strictly increasing property of the function
$\bar\phi_c(r)$ on $\bR_+$. Then, the first required assertion
follows from the fact that $s\mapsto {e^{C_*s}}/{s}$ is strictly
increasing on $\big[ c_2
C_0^{\beta_{1,\phi_c}/(\beta_{2,\phi_c}-1)},\infty \big)$, by choosing
$C_0$ (depending on $C_*$) large enough.

(ii) As seen from (i) and its proof that the function $F_{1,t}(r)$
is strictly increasing on $[C_0\phi_c^{-1}(t),\infty)$ with
$F_{1,t}(\infty)=\infty$. On the other hand, due to the strictly
increasing property of the function $\bar\phi_c(r)$ on $\bR_+$ and
\eqref{e:effdiff}, we know that the function $F_{2,t}(r)$ is
strictly decreasing on $[C_0\phi_c^{-1}(t),\infty)$ with
$F_{2,t}(C_0\phi_c^{-1}(t))=
C^*\bar\phi_c^{-1}(t/(C_0\phi_c^{-1}(t)))/\phi_j^{-1}(t)$
and $F_{2,t}(\infty)=0.$

Furthermore, according to \eqref{polycon} and \eqref{e:effdiff}
again, and the fact that $\phi_c^{-1}(t)\ge \phi_j^{-1}(t)$ for all
$t\in [1,\infty)$, we can obtain that
$$F_{1,t}(C_0\phi_c^{-1}(t))\le c_3$$ and
$$F_{2,t}(C_0\phi_c^{-1}(t))= \frac{C^*\bar\phi_c^{-1}(t/(C_0\phi_c^{-1}(t)))}{\phi_c^{-1}(t)}
\frac{\phi_c^{-1}(t)}{\phi_j^{-1}(t)}\ge
\frac{C^*\bar\phi_c^{-1}(t/(C_0\phi_c^{-1}(t)))}{\phi_c^{-1}(t)}\ge
C^*c_4,$$ where $c_3, c_4$ are independent of $C^*$.

Combining with both conclusions above and taking $C^*=2c_3/c_4$, we
then prove the second desired assertion. \qed
\end{proof}

\noindent{\bf Proof of Proposition \ref{thm:ujeuhkds}.}\quad
We only give the proof for the case of $\beta_*\le 1$.
The proof for the case of $\beta_*>1$ is quite similar.

Let $N\geq 1$ be the positive integer in Lemma \ref{L:lemmastep1}.
We have  by $\VD$ and \eqref{polycon000}  that for $k\geq 1$, there are positive constants
$c_2=c_2(k)$ and $c_3=c_3(k)$ so that
any $t>0$ and $x,y\in M_0$ with $d(x, y)\geq \phi^{-1}(t)$,
\begin{equation}\label{e:3.2}\begin{split}
\frac{1}{V(x,\phi^{-1}(t))}\exp\left(-\frac{a
d(x,y)^{1/N}}{\phi^{-1}(t)^{1/N}}\right)
&\le \frac{c_1}{V(x,d(x,y))}\left(\frac{d(x,y)}{\phi^{-1}(t)}\right)^{d_2}\exp\left(-\frac{a d(x,y)^{1/N}}{\phi^{-1}(t)^{1/N}}\right)\\
&\le \frac{c_2  }{V(x,d(x,y))}\left(\frac{\phi^{-1}(t)}{d(x,y)}\right)^{k \beta_{2,\phi}}
\\
&\le \frac{c_3\, t^k }{V(x,d(x,y))\phi (d(x,y))^k},
\end{split}
\end{equation}
where in
the second inequality we used the fact  that
$r^{d_2+ k \beta_{2, \phi}} e^{-ar^{1/N}/2}\leq c_4(k)$   for  every $  r\geq 1$.

Suppose first that $\beta^*\le 1$. Then $\phi(r)=\phi_j(r)$ for all $r>0$.
By Lemma \ref{L:lemmastep1}, there
exist constants $a, c>0$ and $N\in \bN$ such that for all $x,y\in
M_0$ and
$t>0$,
$$p(t,x,y)\le \frac{ct}{V(x,d(x,y))\phi_j(d(x,y))}+ \frac{c}{V(x,\phi_j^{-1}(t))}\exp\left(-\frac{a d(x,y)^{1/N}}{\phi_j^{-1}(t)^{1/N}}\right).$$
Hence we have by \eqref{e:3.2} with $k=1$   and $\UHKD (\phi)$ that
for every $x,y\in M_0$ and
$t>0$,
$$
p(t,x,y)\le c_5 \left( \frac1{V(x, \phi_j^{-1}(t))} \wedge \frac{t}{V(x,d(x,y))\phi_j(d(x,y))}
\right),
$$
 which proves $\UHK(\phi_j,\phi_c)$ (i.e., $\UHK(\phi_j)$ in \cite{CKW1}).

Now we consider the case of $\beta^*>1$. We split it into two cases.

(1) {\bf Case: $0<t\le 1$.}\,\,
Since $\beta_*\leq 1$, $\phi (r)=\phi_j (r)$ for $r\in [0, 1]$.
By \eqref{polycon} and \eqref{polycon000}, there is some $k_0\geq 1$ so that
$\phi (r)^{k_0} \geq c_6 \phi_j(r)$
on $[1, \infty)$.
We have by Lemma \ref{L:lemmastep1} and \eqref{e:3.2} with $k=k_0$ when $d(x, y)\geq 1$
and with $k=1$ when $1\geq d(x, y) \geq \phi^{-1} (t)=\phi_j^{-1}(t)$
that  for
all $t\in (0, 1]$ and $x,y\in M_0$ with $d(x,y)\ge \phi_j^{-1}(t)$,
$$
p(t,x,y)\le \frac{c_6t}{V(x,d(x,y))\phi_j(d(x,y))},
$$
This together with  $\UHKD(\phi)$ establishes
$\UHK(\phi_j,\phi_c)$ when $0<t\le 1$.

(2) {\bf Case: $t\ge1$.}\,\, By $\UHKD(\phi)$, it suffices
to consider the case that $x,y\in M_0$ and $t\ge 1$ with $d(x,y)\ge
\phi^{-1}(t)=\phi_c^{-1}(t)$. According to Lemma
\ref{L:selftail}, it is enough to verify that there exist
constants $\eta\in (0,\beta_{1,\phi_j}]$ and $c_1,c_2>0$ such that
for any $x\in M_0$, $t\ge1$ and $r>0$,
 \begin{equation}\label{e:diffjump1}\int_{B(x,r)^c}p(t,x,y)\,\mu(dy)\le c_1\left(\frac{\phi_j^{-1}(t)}{r}\right)^{\eta}
 +c_1\exp\left(-c_2\frac{r}{\bar \phi_c^{-1}(t/r)}\right).\end{equation}
 Indeed, by \eqref{e:diffjump1}, Lemma \ref{L:selftail} with $f(r,t)=r/\bar\phi_c^{-1}(t/r)$, we have that for any $x,y\in M_0$ and $t\ge 1$,
 $$p(t,x,y)\le c_{3} \left(\frac{t}{V(x,d(x,y))\phi_j(d(x,y))}+\frac{1}{V(x,\phi_c^{-1}(t))} \exp\left(- \frac{c_0^*d(x,y)}{\bar \phi_c^{-1}(t/d(x,y))}\right)\right).$$
 Therefore, combining both conclusions above, we can obtain the desired assertion. In the following, we will prove \eqref{e:diffjump1}. For this, we will consider four different cases.

(i) If $r\le C_0\phi_c^{-1}(t)$ for some $C_0>0$ which is determined
in the step (ii) below, then, by the non-decreasing property of
$\bar \phi_c$, \eqref{e:1.10} and \eqref{e:effdiff},
 \begin{align*}\exp\left(c_2\frac{r}{\bar \phi_c^{-1}(t/r)}\right)\le &\exp\left(c_2\frac{C_0\phi_c^{-1}(t)}{\bar \phi_c^{-1}(t/(C_0\phi_c^{-1}(t)))}\right)\le e^{c_4}.\end{align*} Hence, \eqref{e:diffjump1} holds trivially by taking $c_1=e^{c_4}.$

(ii) Suppose that $C_0\phi_c^{-1}(t)\le r\le r_*(t)$, where $r_*(t)$
is determined later. For any $n\in \bN\cup\{0\}$, define
$$\rho_n=c_*2^{n\alpha}\bar\phi_c^{-1}\left(\frac{ t}{r}\right)$$ with $d_2/(d_2+\beta_{1,\phi_j})<\alpha<1$, where $c_*$ is also determined later. Then, by
 \eqref{e:trun-org} and Lemma \ref{P:truncated}, for any $x,y\in M_0$ with
$2^nr\le d(x,y)\le 2^{n+1}r$ and $t\ge 1$,
\begin{align*}p(t,x,y)\le & \frac{c_{1}}{V(x,\phi_c^{-1}(t))}\left(1+\frac{2^{n+1}r}{\phi_c^{-1}(t)}\right)^{d_2}\exp\left(\frac{c_{1} t}{\phi(\rho_n)}-\frac{c_{2} 2^nr}{\rho_n}\right)\\
&+\frac{c_{1} t}{V(x,\rho_n)\phi_j(\rho_n)}\exp\left(\frac{c_{1}t}{\phi(\rho_n)}\right)\\
\le &\frac{c_{1}}{V(x,\phi_c^{-1}(t))}
\left(1+\frac{2^{n+1}r}{\phi_c^{-1}(t)}\right)^{d_2}\exp\left(\frac{c_{1} t}{\phi_c(\rho_n)}-\frac{c_{2} 2^nr}{\rho_n}\right)\\
&+\frac{c_{1}t}{V(x,\rho_n)
\phi_j(\rho_n)}\exp\left(\frac{c_{1}t}{\phi_c(\rho_n)}\right),\end{align*}
where in the last inequality we used the facts that if $0<\rho_n\le1
$, then $\phi(\rho_n)=\phi_j(\rho_n)\ge \phi_c(\rho_n)$, and if
$\rho_n\ge1$, then $\phi(\rho_n)=\phi_c(\rho_n)$. Hence, for any
$x\in M_0$, $t\ge 1$ and $C_0\phi_c^{-1}(t)\le r\le r_*(t)$,
\begin{align*}&\int_{B(x,r)^c} p(t,x,y)\,\mu(dy)\\
&=\sum_{n=0}^\infty \int_{B(x,2^{n+1}r)\backslash B(x,2^nr)}p(t,x,y)\,\mu(dy)\\
&\le\sum_{n=0}^\infty\frac{c_{3}
V(x,2^{n+1}r)}{V(x,\phi_c^{-1}(t))}\left(1+\frac{2^{n+1}r}{\phi_c^{-1}(t)}\right)^{d_2}\exp\left(\frac{c_5t}{c_*^{\beta_{1,\phi_c}}2^{n\alpha\beta_{1,\phi_c}} \phi_c(\bar \phi_c^{-1}(t/r))}-\frac{c_22^{n(1-\alpha)} r}{c_*\bar\phi_c^{-1}(t/r)}\right) \\
&\quad+\sum_{n=0}^\infty \frac{c_{4} t V(x,2^{n+1}r)}{V(x,\rho_n)\phi_j(\rho_n)}\exp\left(\frac{c_5t}{c_*^{\beta_{1,\phi_c}}2^{n\alpha\beta_{1,\phi_c}} \phi_c(\bar \phi_c^{-1}(t/r))}\right)\\
&=:I_1+I_2, \end{align*} where the inequality above follows from \eqref{polycon}, and all the constants above are independent of $c_*$. On the one hand,
by taking
$c_*=(1+(2c_1c_5/c_2))^{1/(\beta_{1,\phi_c}-1)}$ (with $c_1\ge1$ being
the constant in \eqref{e:1.10}), we get by $\VD$ and \eqref{e:1.10}
that
\begin{align*}I_1\le & c_{6}\sum_{n=0}^\infty \left(\frac{2^nr}{\phi_c^{-1}(t)}\right)^{2d_2} \exp\left(-c_{7}2^{n(1-\alpha)}\frac{r}{\bar\phi_c^{-1}(t/r)}\right)\le c_{8}\exp\left(-c_{9}\frac{r}{\bar\phi_c^{-1}(t/r)}\right). \end{align*} Here and in what follows, the constants will depend on $c_*$. On the other hand, according to $\VD$, \eqref{polycon} and \eqref{e:1.10} again,
\begin{align*}I_2&\le c_{10}\exp\left(\frac{c_{11}r}{\bar \phi_c^{-1}(t/r)}\right)\sum_{n=0}^\infty \left(\frac{2^n r}{\rho_n}\right)^{d_2} \frac{\phi_j(r)}{\phi_j(\rho_n)}\frac{t}{\phi_j(r)}\\
&\le c_{12}\exp\left(\frac{c_{11}r}{\bar \phi_c^{-1}(t/r)}\right)\left(\frac{r}{\bar\phi_c^{-1}(t/r)}\right)^{d_2+\beta_{2,\phi_j}}\frac{t}{\phi_j(r)}\sum_{n=0}^\infty
2^{n(d_2 -\alpha (d_2+\beta_{1,\phi_j}))}\\
&\le c_{13}\exp\left(\frac{c_{14}r}{\bar \phi_c^{-1}(t/r)}\right)\frac{t}{\phi_j(r)}, \end{align*} where in the last inequality we used the condition $d_2/(d_2+\beta_{1,\phi_j})<\alpha<1$ and the fact that
$$r^{d_2+\beta_{2,\phi_j}}\le c_{15} e^r
\quad \hbox{for all } r\ge c_{16}>0.$$

We emphasize that the argument up to here is
independent of the definition of $r_*(t)$ and the choice of the
constant $C_0$. In view of Lemma \ref{L:keyexist}, we
can find constants $C_0, c_{17}>0$ such that
\begin{equation}\label{e:constant1-}\exp\left(\frac{2c_{14}r}{\beta_{1,\phi_j}\bar \phi_c^{-1}(t/r)}\right)
\le \frac{c_{17} r}{\phi_j^{-1}(t)} \quad \hbox{when } C_0\phi_c^{-1}(t)\le r\le
r_*(t),
\end{equation}
$$
\exp\left(\frac{2c_{14}r}{\beta_{1,\phi_j}\bar
\phi_c^{-1}(t/r)}\right)\ge \frac{c_{17} r}{\phi_j^{-1}(t)} \quad \hbox{for }
r\ge r_*(t),
$$
and
\begin{equation}\label{e:constant1--}
\exp\left(\frac{2c_{14}r_*(t)}{\beta_{1,\phi_j}\bar
\phi_c^{-1}(t/r_*(t))}\right)= \frac{c_{17}
r_*(t)}{\phi_j^{-1}(t)}.
\end{equation}
(Here, we note that, without
loss of generality we may and do assume that
$\frac{2c_{14}}{\beta_{1,\phi_j}}\ge1$.) Then, due to
\eqref{polycon}
and \eqref{e:constant1-},
$$I_2\le c_{18}\left(\frac{r}{\phi_j^{-1}(t)}\right)^{\beta_{1,\phi_j}/2}\left(\frac{t}{\phi_j(r)}\right)\le c_{19}\left(\frac{\phi_j^{-1}(t)}{r}\right)^{\beta_{1,\phi_j}/2}.$$
Putting $I_1$ and $I_2$ together, we arrive at
\eqref{e:diffjump1}.

For the need in steps (iii) and (iv) later, we will consider bounds
for $r_*(t)$. We first treat the lower bound for $r_*(t)$.
Let $r\ge C_0\phi_c^{-1}(t)$. By
\eqref{polycon} and \eqref{e:effdiff}, we have
\begin{equation}\label{e:lowerbounduse}\exp\left(\frac{2c_{14} r}{\beta_{1,\phi_j}\bar\phi_c^{-1}(t/r)}\right)
\le
\exp\left(c_{20}\left(\frac{r}{\phi_c^{-1}(t)}\right)^{\beta_{2,\phi_c}/(\beta_{1,\phi_c}-1)}\right).\end{equation}
Hence, \eqref{e:constant1-} holds, if
$$\exp\left(c_{20}\left(\frac{r}{\phi_c^{-1}(t)}\right)^{\beta_{2,\phi_c}/(\beta_{1,\phi_c}-1)}\right)\le \frac{c_{17} r}{\phi_j^{-1}(t)},\quad
C_0\phi_c^{-1}(t)\le r\le r_*(t);$$
namely,
\begin{align*}&\log c_{17}+ \log \frac{r}{\phi_c^{-1}(t)}+ \log \frac{\phi_c^{-1}(t)}
{\phi_j^{-1}(t)}\ge
c_{20}\left(\frac{r}{\phi_c^{-1}(t)}\right)^{\beta_{2,\phi_c}/(\beta_{1,\phi_c}-1)}\end{align*}
holds for all $ C_0\phi_c^{-1}(t)\le r\le r_*(t)$.
From this, one can
see that $$r_*(t)\ge C_1
\phi_c^{-1}(t)\log^{(\beta_{1,\phi_c}-1)/\beta_{2,\phi_c}}(\phi^{-1}_c(t)/\phi^{-1}_j(t))$$
for some constant $C_1>0$ which is independent of $t$.
For the upper
bound for $r_*(t)$, similar to the argument for
\eqref{e:lowerbounduse}, we have
$$\exp\left(\frac{2c_{14} r}{\beta_{1,\phi_j}\bar\phi_c^{-1}(t/r)}\right)
\ge
\exp\left(c_{21}\left(\frac{r}{\phi_c^{-1}(t)}\right)^{\beta_{1,\phi_c}/(\beta_{2,\phi_c}-1)}\right).$$
Hence, we can find that $$r_*(t)\le C_2
\phi_c^{-1}(t)\log^{(\beta_{2,\phi_c}-1)/\beta_{1,\phi_c}}(\phi^{-1}_c(t)/\phi^{-1}_j(t)),$$
where $C_2>0$ is also independent of $t$.

(iii)
Suppose that $r\ge r^*(t):=C_3
\phi_c^{-1}(t)\log^{N}(\phi^{-1}_c(t)/\phi^{-1}_j(t))$, where $N\in
\bN$ is given in Lemma \ref{L:lemmastep1}, and $C_3>0$ is determined
later.
 According to Lemma \ref{L:lemmastep1}, there exist constants $a, c>0$ and $N\in \bN$ such that for all $x,y\in M_0$ and $t\ge 1$ with $d(x,y)\ge r^*(t)$,
\begin{equation}\label{eq:19-6-22}\begin{split}p(t,x,y)\le &\frac{ct}{V(x,d(x,y))\phi_j(d(x,y))}+ \frac{c}{V(x,\phi_c^{-1}(t))}\exp\left(-\frac{a d(x,y)^{1/N}}{\phi_c^{-1}(t)^{1/N}}\right) \\
\le&\frac{ct}{V(x,d(x,y))\phi_j(d(x,y))}+ \frac{c_1}{V(x,d(x,y))}\left(\frac{d(x,y)}{\phi_c^{-1}(t)}\right)^{d_2}\exp\left(-\frac{a d(x,y)^{1/N}}{\phi_c^{-1}(t)^{1/N}}\right) \\
\le& \frac{ct}{V(x,d(x,y))\phi_j(d(x,y))}+ \frac{c_2}{V(x,d(x,y))}\exp\left(-\frac{a}{2}\frac{ d(x,y)^{1/N}}{\phi_c^{-1}(t)^{1/N}}\right),
\end{split}\end{equation}
where in the second inequality we used $\VD$ and the last inequality follows from the fact that
$$r^{d_2}\le c_3 e^{ar^{1/N}/2},\quad r\ge c_4>0.$$
Without loss of generality, we may and do assume that $N\ge
(\beta_{2,\phi_c}-1)/\beta_{1,\phi_c}$. In particular, $r^*(t)\ge
r_*(t)$ by the upper bond for $r_*(t)$ mentioned at the end of step
(ii) and also by choosing $C_3>0$ large enough in the definition of
$r^*(t)$ if necessary.

 Next, suppose
that there is a constant $C_3>0$ in the definition of $r^*(t)$ such
that
\begin{equation}\label{e:constant1+}\exp\left(\frac{a}{2}\frac{r^{1/N}}{\phi_c^{-1}(t)^{1/N}}\right)\ge\frac{ c_5 \phi_j(r)}{t},\quad r\ge r^*(t)\end{equation}
holds for some $c_5>0$.
Then, by \eqref{eq:19-6-22} we have
for all $x,y\in M_0$
and $t\ge 1$ with $d(x,y)\ge r^*(t)$,
$$p(t,x,y)\le \frac{c_6t}{V(x,d(x,y))\phi_j(d(x,y))}.$$ Hence, by \eqref{polycon} and Lemma \ref{intelem}, for all $x\in M_0$, $t\ge1$ and
$r\ge r^*(t)$, we obtain
\begin{align*}\int_{B(x,r)^c}p(t,x,y)\,\mu(dy)\le &c_6 \int_{B(x,r)^c}\frac{t}{V(x,d(x,y))\phi_j(d(x,y))}\,\mu(dy)\\
\le & \frac{c_7t}{\phi_j(r)}\le
c_{8}\left(\frac{\phi_j^{-1}(t)}{r}\right)^{\beta_{1,\phi_j}},\end{align*}
proving \eqref{e:diffjump1}.

We now check \eqref{e:constant1+} following the argument of the
lower bound of
$r_*(t)$ in the end of step (ii).
First, note
from \eqref{polycon} that
$$ \frac{ c_5 \phi_j(r)}{t}\le c_9 \left(\frac{r}{\phi_j^{-1}(t)}\right)^{\beta_{2,\phi_j}}.$$ Then, \eqref{e:constant1+} is a consequence of the following inequality
$$\exp\left(\frac{a}{2}\frac{r^{1/N}}{\phi_c^{-1}(t)^{1/N}}\right)\ge c_9 \left(\frac{r}{\phi_j^{-1}(t)}\right)^{\beta_{2,\phi_j}},\quad r\ge r^*(t);$$ that is,
$$ \frac{a}{2}\left(\frac{r}{\phi_c^{-1}(t)}\right)^{1/N}\ge \log c_9+ \beta_{2,\phi_j} \log \frac{r}{\phi_c^{-1}(t)}
+\beta_{2,\phi_j}\log \frac{\phi_c^{-1}(t)}{\phi_j^{-1}(t)},\quad
r\ge r^*(t).$$
From this, we see that
we can take a proper constant $C_3>0$ such
that \eqref{e:constant1+} is satisfied.

(iv) Let $r\in [r_*(t), r^*(t)]$. For any $x\in M_0$, $t\ge 1$ and
$r\in [r_*(t),r^*(t)]$, we find by the conclusion in step (ii) that
\begin{align*}\int_{B(x,r)^c}p(t,x,y)\,\mu(dy)&\le \int_{B(x,r_*(t))^c} p(t,x,y)\,\mu(dy)\\
&\le c_1\left(\frac{\phi_j^{-1}(t)}{r_*(t)}\right)^{\beta_{1,\phi_j}/2}+c_1\exp\left(-c_2\frac{r_*(t)}{\bar\phi_c^{-1}(t/r_*(t))}\right)\\
&=:I_1+I_2.\end{align*}

On the one hand, we find by the facts $r_*(t)\ge C_1 \phi_c^{-1}(t)\log^{(\beta_{1,\phi_c}-1)/\beta_{2,\phi_c}}(\phi^{-1}_c(t)/\phi^{-1}_j(t))$ and $\phi_j^{-1}(t)\le \phi_c^{-1}(t)$ that
\begin{align*}I_1&\le c_3\left(\frac{\phi_j^{-1}(t)}{\phi_c^{-1}(t)\log^{(\beta_{1,\phi_c}-1)/\beta_{2,\phi_c}}(\phi^{-1}_c(t)/\phi^{-1}_j(t))}\right)^{\beta_{1,\phi_j}/2}\\
&\le c_3\left(\frac{\phi_j^{-1}(t)}{\phi_c^{-1}(t)\log^N(\phi^{-1}_c(t)/\phi^{-1}_j(t))}
\right)^{(\beta_{1,\phi_c}-1)\beta_{1,\phi_j}/(2N\beta_{2,\phi_c})}\\
&\le c_4
\left(\frac{\phi_j^{-1}(t)}{r}\right)^{(\beta_{1,\phi_c}-1)\beta_{1,\phi_j}/(2N\beta_{2,\phi_c})}.
\end{align*} On the other hand, without loss of generality we may and do assume that $c_2\in (0,1)$ in the term $I_2$. By
\eqref{e:constant1--} and the fact that we can assume in
\eqref{e:constant1--} that $\frac{2c_{14}}{\beta_{1,\phi_j}}\ge1$,
there is a constant $\theta\in (0,1)$ such that
\begin{align*}I_2&\le c_5\left(\frac{\phi_j^{-1}(t)}{r_*(t)} \right)^{\theta}\le c_6\left(\frac{\phi_j^{-1}(t)}{\phi_c^{-1}(t)\log^{(\beta_{1,\phi_c}-1)/\beta_{2,\phi_c}}(\phi^{-1}_c(t)/\phi^{-1}_j(t))}\right)^{\theta}\\
&\le
c_7\left(\frac{\phi_j^{-1}(t)}{\phi_c^{-1}(t)\log^N(\phi^{-1}_c(t)/\phi^{-1}_j(t))}\right)^{\theta(\beta_{1,\phi_c}-1)/(N\beta_{2,\phi_c})}\le
c_8\left(\frac{\phi_j^{-1}(t)}{r}\right)^{\theta(\beta_{1,\phi_c}-1)/(N\beta_{2,\phi_c})}.\end{align*}

Combining all the estimates above, we can obtain \eqref{e:diffjump1}
in this case with
$$\eta=\min\{(\beta_{1,\phi_c}-1)\beta_{1,\phi_j}/(2N\beta_{2,\phi_c}),
\theta(\beta_{1,\phi_c}-1)/(N\beta_{2,\phi_c})\}.$$ The proof is
complete. \qed

We remark that Lemma \ref{L:lemmastep1} can quickly lead to the following estimate.

\begin{proposition}\label{P:3.8}  Under $\VD$ and
\eqref{polycon}, if $\UHKD(\phi)$, $\J_{\phi,\le}$ and $\E_{\phi}$
hold, then there is a constant $c>0$ such that for any $t>0$ and
$x,y\in M_0$,
\begin{equation}\label{heat:remark1}
p(t,x,y)\le c\left(\frac{1}{V(x,\phi^{-1}(t))} \wedge \frac{t}{V(x,d(x,y))\phi(d(x,y))}\right).
\end{equation}
\end{proposition}

\begin{proof}
According to $\UHKD(\phi)$, we only need to verify that there is a constant $c>0$ such that for any $t>0$ and
$x,y\in M_0$,
$$
p(t,x,y)\le   \frac{ct}{V(x,d(x,y))\phi(d(x,y))}.$$

Following the proof of
\cite[Lemma 4.7]{CKW4},
 we can find that, under the assumptions of this proposition, there exist constants $a, c_0>0$ and $N\in \bN$ such that for all $x,y\in M_0$ and $t>0$,
\begin{equation}\label{e:pppfff}
p(t,x,y)\le \frac{c_0t}{V(x,d(x,y))\phi(d(x,y))}+ \frac{c_0}{V(x,\phi^{-1}(t))}\exp\left(-\frac{a d(x,y)^{1/N}}{\phi^{-1}(t)^{1/N}}\right).\end{equation}

 When $d(x,y) \geq \phi^{-1}(t)$,
we have by \eqref{e:3.2} with $k=1$
that
$$
\frac{1}{V(x,\phi^{-1}(t))}\exp\left(-\frac{a d(x,y)^{1/N}}{\phi^{-1}(t)^{1/N}}\right)
\leq  \frac{c_1t}{V(x,d(x,y))\phi(d(x,y))};
$$
while for $d(x, y) \leq \phi^{-1}(t)$,
$$
\frac{1}{V(x,\phi^{-1}(t))}\exp\left(-\frac{a d(x,y)^{1/N}}{\phi^{-1}(t)^{1/N}}\right)
\leq \frac{c_2}{V(x,\phi^{-1}(t))}
\leq \frac{c_3t}{V(x,d(x,y))\phi(d(x,y))}.
$$
These together with
\eqref{e:pppfff} yield the estimate \eqref{heat:remark1}. \qed
\end{proof}

Obviously, the upper bound estimate \eqref{heat:remark1} is not as sharp as $\UHK(\phi_j,\phi_c)$.
However,
\eqref{heat:remark1} plays a role in the characterizations of parabolic Harnack inequalities below.

\medskip

\noindent{\bf Proof of Theorem \ref{T:1.13}.}
The implications of (1) $\Longrightarrow$ (4) and (4) $\Longrightarrow$ (5) of Theorem
\ref{T:1.13} follow from Propositions \ref{P:2.1}, \ref{l:jk},
 \ref{Conserv} and \ref{P:gdCS},
while the $(5)\Longrightarrow(3)$ part and the $(3)\Longrightarrow (2)$ part follow from Propositions \ref{P:exit} and \ref{ehi-ukdd}.
Note that by the proof of
\cite[Lemma 4.21]{CKW1}, under $\VD$ and
\eqref{polycon},
$\E_\phi$ and $\J_{\phi, \leq}$
imply the conservativeness of $(\sE, \sF)$. This together with Proposition \ref{thm:ujeuhkds}
gives the implication $(2)\Longrightarrow(1)$. This completes the proof of the theorem.
\qed

\section{Characterizations of $\HK_-(\phi_j,\phi_c)$ and $\HK(\phi_j,\phi_c)$} \label{section6}

This section is devoted to the proof of Theorem \ref{T:main}. Since the proof is similar to that of \cite[Theorem 1.14]{CKW4}, we only present the different points.

\subsection{$\PI(\phi)+ \J_{\phi_j} +\CSJ(\phi) \Longrightarrow \HK_-(\phi_j,\phi_c)$}\label{subsection:lower}

\begin{proposition}
\label{P:ehr}
If $\VD$, $\RVD$, \eqref{polycon}, $\PI(\phi)$, $\J_{\phi_j,\le}$ and $\CSJ(\phi)$ hold, then we have $\NDL(\phi)$.
\end{proposition}

\begin{proof}
By the proof of
\cite[Propisition 4.13]{CKW2}, under $\VD$, $\RVD$ and \eqref{polycon}, $$ \PI(\phi)+\J_{\phi,\le}+\CSJ(\phi)\Longrightarrow \EHR.$$ Furthermore, by Theorem \ref{T:1.13}, under $\VD$, $\RVD$ and \eqref{polycon},
$$ \PI(\phi)+\J_{\phi,\le}+\CS(\phi)\Longrightarrow \E_\phi,$$ where we used the fact that $\PI(\phi)$ implies $\FK(\phi)$ under $\RVD$
from Proposition \ref{P:2.1}.
With those two conclusions at hand, the desired assertion follows from the proof of \cite[Proposition 4.9]{CKW2}. \qed \end{proof}

The following statement proves that the $(5)\Longrightarrow(1)$ part in Theorem \ref{T:main}.

 \begin{proposition}\label{ejlhk}
If $\VD$, $\RVD$, \eqref{polycon}, $\PI(\phi)$, $\J_{\phi_j}$ and $\CSJ(\phi)$ hold, then we have $\HK_-(\phi_j,\phi_c)$.
\end{proposition}

\begin{proof} By Proposition
\ref{P:ehr},
we have $\NDL(\phi)$. In particular, for any $x,y\in M_0$ and $t>0$ with $d(x,y)\le c_0\phi^{-1}(t)$ for some constant $c_0>0$,
$$p(t,x,y)\ge \frac{c_1}{V(x,\phi^{-1}(t))}.$$
Following the arguments in step (ii) for the proof of \cite[Proposition 5.4]{CKW1}, we know that for all $x,y\in M_0$ and $t>0$ with $d(x,y)\ge c_0\phi^{-1}(t)$,
$$p(t,x,y)\ge \frac{c_2 t}{V(x,d(x,y))\phi_j(d(x,y))}.$$
Note that, by Theorem \ref{T:1.13}, under assumptions of this proposition, $\UHK(\phi_j,\phi_c)$ holds.
Combining with all the conclusions, we can obtain the desired assertion.
 \qed\end{proof}

\subsection{From $\HK_-(\phi_j,\phi_c)$ to $\HK(\phi_j,\phi_c)$}

\begin{proposition}\label{P:twosided}
If $(M,d,\mu)$ is connected and satisfies the chain condition, then
$\HK_- (\phi_j,\phi_c)$ implies
$\HK(\phi_j,\phi_c)$.
In particular, if
$\VD$, $\RVD$, \eqref{polycon}, $\PI(\phi)$, $\J_{\phi_j}$ and $\CSJ(\phi)$ hold,
then so does $\HK(\phi_j,\phi_c)$.

\end{proposition}

\begin{proof}
We only give a proof for the case of $\beta_*\le 1$. The case of $\beta_*>1$ can be proved
in a similar way.
By the definition of $\HK_-(\phi_j,\phi_c)$ and
Remark \ref{R:1.16}(ii),  we only need to
establish the lower bound of $\HK (\phi_j, \phi_c)$ in the case  of $\beta^*>1$ and $\phi_c (r)$
is not comparable to $\phi_j (r)$ on $[1, \infty)$ for
  $t\in [1,\infty)$ and $x,y\in M_0$ with $d(x,y)\ge c_0\phi_c^{-1}(t)$ for some constant $c_0>0$.
In this case,  a more explicit expression for the lower bound estimate of $\HK (\phi_j, \phi_c)$
is given in
Remark \ref{R:1.22}(iii) and we only need to establish the exponential part.  		
This can be done by using
the standard chaining argument (see the proof of \cite[Proposition 5.2(i)]{BGK2} or that of
\cite[Proposition 5.6]{CKW4}).
We omit the details here.
The second assertion follows from the first one and Theorem \ref{T:main}.
\qed\end{proof}

\noindent{\bf Proof of Theorem \ref{T:main}}
 By Proposition \ref{l:jk}, $(1)\Longrightarrow (2)$. $(2)\Longrightarrow(3)$ follows from the remark after Definition \ref{D:1.11}. According to \cite[Proposition 3.5]{CKW2}, under $\VD$ and \eqref{polycon}, $\NDL(\phi)$ implies $\PI(\phi)$, and so $\FK(\phi)$ under $\RVD$ by Proposition \ref{P:2.1};
while $\NDL(\phi)$ also implies $\E_\phi$
thanks to $\RVD$ again.
With these at hand, by Theorem \ref{T:1.13}, we have $(3)\Longrightarrow(5)$. $(5)\Longrightarrow(6)$
can be seen by Proposition \ref{P:gdCS}. Furthermore, by Proposition \ref{ejlhk}, $(6)\Longrightarrow(1)$.
The equivalence between (4) and (6) follows from Propositions \ref{P:2.1} and \ref{ehi-ukdd},
and the $(2) \Longleftrightarrow (4)$  part of Theorem \ref{T:1.13}.
When the space $(M,d,\mu)$ is connected and satisfies the chain condition,
 $(5) \Longrightarrow (7)$ has been proven in Proposition \ref{P:twosided}. Clearly, $(7) \Longrightarrow(1)$. The proof is complete. \qed

\section{Applications}\label{section7}
\subsection{Parabolic Harnack inequalities}

In this subsection, we first present a proof for Theorem \ref{T:PHI} on the stable characterization of parabolic Harnack inequalities.

\noindent{\bf Proof of Theorem \ref{T:PHI}.}
In view of Proposition \ref{P:3.8},
the first assertion of Theorem \ref{T:PHI}
can be established by the same
the arguments in \cite[Subsection 4.3]{CKW2}. See also \cite[Section 6]{CKW4} for more details.
The second one is a direct consequence of the first one and Theorem \ref{T:main}.
Indeed, similar to \cite[Section 6]{CKW4}, one can also derive several characterizations for $\PHI(\phi)$ in the present setting. The details are omitted here.
We emphasize that, different from the pure non-local setting as
studied in \cite{CKW2},
when $J(x, y)$ has lighter tails at infinity, $\PHI(\phi)$ only implies $\J_{\phi,\le }$
but not  $\J_{\phi_j,\le}$ in general.
See the example below.
\qed

\begin{example}\label{PHIexm} {\bf ($\PHI(\phi)$ alone does not imply $\J_{\phi_j,\le}$)}\,\,\rm Let $M=\bR^d$, and
$$J(x,y)\asymp \begin{cases} \frac{1}{|x-y|^{d+\alpha}},&\quad |x-y|\le 1; \\
 \frac{1}{|x-y|^{d+\beta}}, &\quad |x-y|\ge1,\end{cases}$$ where $\alpha\in (0,2)$ and $\beta\in (2,\infty)$. We consider the following regular Dirichlet form
 $$\sE(f,g)=\iint_{\bR^d\times \bR^d} (f(x)-f(y))(g(x)-g(y))J(x,y)\,dx\,dy$$ and
 $\sF=\overline{C_1(\bR^d)}^{\sE_1}$. It has been proven in
\cite[Theorem 3.10]{BKKL1}
that $\PHI(\phi)$ holds with $\phi(r)=r^\alpha\vee r^2$. Therefore, for this example $\PHI(\phi)$ holds for any $\beta\in (2,\infty)$; that is, $\PHI(\phi)$ alone does not imply upper bounds of the jumping kernel.\end{example}

\subsection{Transferring method for heat kernel estimates and parabolic Harnack inequalities}

By the stability results in Theorems \ref{T:main}, \ref{T:1.13} and
\ref{T:PHI}, we can obtain heat kernel estimates and parabolic
Harnack inequalities for a large class of symmetric jump processes
using ``transferring method". Namely, we first establish heat
kernel estimates and parabolic Harnack inequalities for
a jump process with a particular jumping kernel $J(x, y)$,
and then use Theorems
\ref{T:main}, \ref{T:1.13} and \ref{T:PHI} to obtain heat kernel
estimates and parabolic Harnack inequalities for other symmetric
jump processes whose jumping kernels are comparable to $J(x, y)$.
For instance, Claim 2 in Example $\ref{Exam-M}$ gives heat
kernel estimates (hence parabolic Harnack inequalities as well) for the subordinate
process $Y$. Thus, by Theorems \ref{T:main} and \ref{T:PHI}, we can obtain
the heat kernel estimates (and the parabolic Harnack inequalities) for pure jump
processes whose jumping kernels enjoy \eqref{eq:jownex1}.
In \cite[Section 6.1]{CKW1} and \cite[Section 5]{CKW2}, we have other examples on fractals
that support anomalous diffusions with two-sided heat kernel estimates.
The subordination
of these diffusion processes enjoy $\HK (\phi)$ and hence $\PHI (\phi)$, which
can be served as the base examples. See also \cite[Section 7]{CKW4}
for related discussions about diffusions with jumps.

\subsection{Examples}
 To illustrate the main results of this paper, we present two more  examples in this subsection.
 we will mainly concentrate on the case of
$$
\phi (r)=\phi_j (r){\bf1}_{[0, 1]}+\phi_c (r) {\bf1}_{(1, \infty)}.
$$
This corresponds to the cases of $\beta_*\le 1$ and $\beta^*>1$, and of $\beta_*> 1$ with $\phi_c(r)=\phi_j(r)$ on $[0,1]$ and $\beta^*>1$. One can use the similar ideas to give examples in other cases.

The first
example shows that for our main results,
the scale function $\phi_c(r)$ does not need to be the scaling function
corresponding to diffusion processes (e.g.\ see Example
\ref{Exam-M}) even in
the case of the Euclidean space.

\begin{example}\label{e:example5.2}\rm Let $M=\bR^d$ and $\mu(dx)=dx$. Consider the following jumping kernel
$$
J(x,y)\simeq \frac{1}{|x-y|^d\phi_j(|x-y|)},$$ where $\phi_j$ is a strictly increasing continuous function with $\phi_j(0)=0$ and $\phi_j(1)=1$ so that
\begin{itemize}
\item[${\rm(i)}$] there are constants $c_{1,\phi_j}, c_{2,\phi_j }>0$ and $0<\beta_{1,\phi_j }\le \beta_{2,\phi_j }<2$ such that
\begin{equation}\label{eq:624-1BF}
c_{1,\phi_j } \Big(\frac Rr\Big)^{\beta_{1,\phi_j}} \leq
\frac{\phi_j (R)}{\phi_j (r)} \ \leq \ c_{2,\phi_j} \Big(\frac
Rr\Big)^{\beta_{2,\phi_j}}
\quad \hbox{for all }
0<r \le R\le 1,
\end{equation}

\item[${\rm(ii)}$] there are constants $c^*_{1,\phi_j }, c^*_{2,\phi_j }>0$ and
 $1<\beta^*_{1,\phi_j }\le \beta^*_{2,\phi_j }<\infty$ such that
\begin{equation}\label{eq:624-2BF}
c^*_{1,\phi_j} \Big(\frac Rr\Big)^{\beta^*_{1,\phi_j}} \leq
\frac{\phi_j (R)}{\phi_j (r)} \ \leq \ c^*_{2,\phi_j} \Big(\frac
Rr\Big)^{\beta^*_{2,\phi_j}}
\quad \hbox{for all } 1\le r \le R.
\end{equation}
 \end{itemize}

Let $(\sE, \sF)$ be the associated pure jump Dirichlet form with jumping kernel $J(x,y)$ as above on $L^2(\bR^d;dx)$, where $\sF=\{f\in L^2(\bR^d;dx): \sE(f,f)<\infty\}$.
It is easy to check (see \cite[p.\ 1969]{CKK2}) that  $(\sE, \sF)$ is a regular Dirichlet form on $L^2(\bR^d; dx)$.
By (i), for every $r\in (0, 1]$,
\begin{equation}\label{e:5.3}
\begin{split}
\int_0^r \frac{s}{\phi_j(s)}\,ds
&=  \sum_{n=1}^\infty \int_{2^{-n}r}^{2^{-n+1}r}\frac{s}{\phi_j(s)}\,ds \le \sum_{n=1}^\infty
\frac{2^{-n+1}r}{\phi_j(2^{-n} r)} 2^{-n}r   \\
&\le  2c_{2,\phi_j}
\frac{r^2}{\phi_j(r)}\sum_{n=1}^\infty 2^{-n(2-\beta_{2,\phi_j})}
\leq \frac{c_3\, r^2}{\phi_j (r)}.
\end{split}\end{equation}
On the other hand,
\begin{equation}\label{e:5.4}
\int_0^r \frac{s}{\phi_j(s)}\,ds\geq \frac1{\phi_j (r)} \int_0^r s \,ds =
\frac{r^2}{2 \phi_j (r)}  \quad \hbox{for every } r>0.
\end{equation}

Define
$$
\Phi(r)=\frac{r^2}{2\int_0^r {s}/{\phi_j(s)}\,ds},\quad r>0.
$$
Since
$$
\frac{d}{dr} \Phi (r)= \frac{2 r \left( 2\int_0^r s/
\phi_j (s)\, ds - r^2/\phi_j (r) \right)} { 4 \left(\int_0^r s/\phi_j (s) \,ds \right)^2} >0,
$$
 $\Phi(r)$ is strictly increasing on $\bR_+$.
 Since $\frac{r^2}{\Phi(r)}= 2\int_0^r {s}/{\phi_j(s)}\,ds$ is an increasing function,
\begin{equation}\label{e:Phi-scaling}
\frac{\Phi(R)}{\Phi(r)}\le \left(\frac{R}{r}\right)^2 \quad \hbox{for all } R\geq r>0.
\end{equation}
Moreover,  by \eqref{e:5.3} and \eqref{e:5.4},
 $\Phi (r) \leq \phi_j(r)$ on $[0, \infty)$, and $\Phi (r)$ is comparable to $\phi_j(r)$ on $[0, 1]$.

Define
$$
\phi_c(r)=\begin{cases}
r^2, \qquad & r\in (0, 1], \\
\frac{\Phi (r)}{\Phi(1)}, & r\in [1, \infty).
\end{cases}
$$
By condition (i), $\phi_j(r)\ge c_{2,\phi_j}^{-1}r^{\beta_{2,\phi_j}}$ for all $r\in (0,1]$ with $\beta_{2,\phi_j} \in (0, 2)$. So there is a constant $c_0>0$ so that $\phi_c \leq c_0 \phi_j$
on $[0, \infty)$.
Clearly, $\phi_c$ is a strictly increasing and continuous function on $\bR_+$.
We claim that there are constants $c_{1, \phi_c} >0$ and
$ \beta_{1,\phi_c} \in (1, 2]$ such that
\begin{equation}\label{e:5.6}
c_{1,\phi_c} \Big(\frac Rr\Big)^{\beta_{1,\phi_c}} \leq
\frac{\phi_c (R)}{\phi_c (r)} \leq \Big(\frac Rr\Big)^{2}
\quad \hbox{for all }
R\geq r>0.
\end{equation}
Note that the upper bound holds by \eqref{e:Phi-scaling}, and the lower bounds
holds for $0<r\leq R\leq 1$ as $\phi_c(r)=r^2$ on $[0, 1]$.
Below, without loss of generality
we assume $\beta_{1,\phi_j}^*<2$
(otherwise, we replace $\beta_{1,\phi_j}^*$ with $\beta_{1,\phi_j}^*:=2-\varepsilon$
for small $\varepsilon\in (0,1)$, and (ii) still holds).
 For any $n\ge1$ and
$r\geq 1$,
by a change of variable and (ii),
\begin{align*}
\int_0^{2^nr}\frac{s}{\phi_j(s)}\,ds &=\int_0^r\frac{s}{\phi_j(s)}\,ds+\sum_{k=1}^n\int_{2^{k-1}r}^{2^kr} \frac{s}{\phi_j(s)}\,ds\\
&\le \int_0^r\frac{s}{\phi_j(s)}\,ds+\sum_{k=1}^n 2^{2(k-1)}\int_{r}^{2r}\frac{s}{\phi_j(2^{k-1}s)}\,ds\\
&\le \int_0^r\frac{s}{\phi_j(s)}\,ds+ \frac 1{c_{1,\phi_j}^*}\sum_{k=1}^n2^{(k-1)(2-\beta_{1,\phi_j}^*)}\int_{r}^{2r} \frac{s}{\phi_j(s)}\,ds\\
&\le \left(1\vee \frac 1{c_{1,\phi_j}^*}\right)\sum_{k=0}^n2^{(k-1)(2-\beta_{1,\phi_j}^*)}\int_r^{2r}\frac{s}{\phi_j(r)}\,ds\\
&\le c_4 2^{n(2-\beta_{1,\phi_j}^*)}\int_0^r\frac{s}{\phi_j(s)}\,ds,
\end{align*}
(note that $\beta_{1,\phi_j}^*<2$ is used in the last inequality), and so
$$
\phi_c(2^nr)=\frac{2^{2n}r^2}
{2\Phi(1)\int_0^{2^nr}{s}/{\phi_j(s)}\,ds}\geq c_5
2^{n \beta^*_{1,\phi_j} } \phi_c(r).
$$
 Here the constants $c_4,c_5>0$ are
 independent of $r,R$ and $n$.
 So for any
$R>r\geq 1$,
$$\phi_c(R)\ge \phi_c(2^{[\log_2(R/r)]}r)\ge
c_5
2^{[\log_2(R/r)] \, \beta_{1,\phi_j}^*} \phi_c(r)
\ge c_6
\left(\frac R r\right)^{\beta_{1,\phi_j}^*} \phi_c(r).
$$
For $0<r\le 1\le R$, we have by above
and that $\phi_c(r)=r^2$
on $[0, 1]$,
   $$\frac{\phi_c(R)}{\phi_c(r)}=\frac{\phi_c(R)}{\phi_c(1)}\frac{\phi_c(1)}{r^2}
	\ge c_6 \frac{R^{\beta_{1,\phi_j}^*}}{r^2} \ge c_6
	\left(\frac R r\right)^{\beta^*_{1,\phi_j}}.$$
The claim \eqref{e:5.6} is proved.

As noted earlier, $\Phi (r)$ and $\phi_j(r)$ are comparable on $[0, 1]$.
Consequently,
$$
\phi (r):= {\bf1}_{[0,
1]}(r) \phi_j (r) + {\bf1}_{(1, \infty)} (r) \phi_c (r)
$$
is comparable to $\Phi (r)$ on $[0, \infty)$. Hence $\PI (\phi)$ holds
by \cite[Proposition 3.2]{BKKL1}. On the other hand, by \cite[Lemma
3.9]{BKKL1}, $\EP_{\phi,\le}$ is
satisfied. So  $\CSJ(\phi)$ holds by the proofs of Propositions \ref{Conserv} and \ref{P:gdCS}.
Therefore,
$\HK(\phi_j,\phi_c)$ holds for $(\sE, \sF)$ by Theorem \ref{T:main}.
This assertion improves
\cite[Theorem 1.4 and Corollary 1.5]{BKKL1},
where an extra condition that
$\beta_{1, \phi_j}>1$
  is required.

 We also note that in this example,
 $\phi_c (r)$ does not need to be comparable to the quadratic function $r\mapsto r^2$ for all $r>0$.
 Nevertheless, we can treat the case that pure jump Dirichlet forms
with the growth order of $\phi_j(r)$
 not necessarily strictly less than $2$, as in \cite{BKKL1}.
By \cite[Corollaries 2.6.2 and 2.6.4]{BGT},
$\phi_c(r)$
and $\phi_j(r)$ are comparable on $[1,\infty)$ if and only if $\beta_{2,\beta_j}^*<2$,
in which case, the
heat kernel estimate $\HK (\phi_j, \phi_c)$ is reduced to $\HK (\phi_j)$
 in \cite{CKW1} (see Remark \ref{R:1.16}(ii)).
Therefore, the \lq \lq diffusive scaling\rq\rq\, appears in
$\HK (\phi_j, \phi_c)$ only when $\beta^*_{2,\phi_j}\ge2$.
For example, when $\phi_j(r)=r^\alpha\vee r^2$ for all $r>0$ with $\alpha\in (0,2)$,
we can take
$$
\phi_c(r)
:= r^2
{\bf 1}_{\{0\le r\le 1\}}+ \frac{r^2}{\log(1+r)}{\bf 1}_{\{r>1\}}
$$
and so
$$
\phi(r) = r^\alpha {\bf 1}_{\{0\le r\le 1\}}+ \frac{r^2}{\log(1+r)}{\bf 1}_{\{r>1\}}.
$$
It holds that
for any $t\ge 1$ and $x,y\in \bR^d$,
$$p^{(c)}(t,x,y)\asymp
\frac{1}{V(x,(t\log(1+t))^{1/2})}\exp\left(-\frac{d(x,y)^2}{t\log(1+t/ d(x,y))}\right).
$$
 This does not belong to the so-called (sub)-Gaussian estimates. Indeed
long time behavior of the process is super-diffusive, and its heat kernel estimates are \lq\lq super-Gaussian".
\end{example}

Our second example
is to illustrate
 our stable characterization of parabolic Harnack inequalities. The assertion of the example below is a counterpart of \cite[Theorem 1.4]{CKK2},
 which only concerns
 a local version of parabolic Harnack inequalities. One can use the assertion below to
 study parabolic Harnack inequalities directly for
 a large class of symmetric jump processes in \cite{CKK3}.

\begin{example}\label{Exm-phi}\rm
Let $M=\bR^d$ and $\mu(dx)=dx$. Consider the non-local (regular) Dirichlet form $(\sE, \sF)$ on $L^2(\bR^d;dx)$ given by
\begin{align*}\sE(f,f)=&\int_{\bR^d\times \bR^d} (f(x)-f(y))^2J(x,y)\,dx\,dy,\\
\sF=&\left\{f\in L^2(\bR^d;dx): \sE(f,f)<\infty\right\},\end{align*} where
$J(x,y)$ is a non-negative symmetric function on $\bR^d\times \bR^d$ such that
\begin{align*}
J(x,y)\simeq\frac{1}{|x-y|^d\phi_j(|x-y|)},&\qquad |x-y|\le 1,\\
J(x,y) \preceq \frac{1}{|x-y|^{d+2}},\quad\quad\quad\,\,\,\, &\qquad |x-y|\ge 1
\end{align*}
and
\begin{equation}\label{e:5.7}
\sup_{x\in \bR^d}\int_{\{|x-y|\ge1\}}|x-y|^2J(x,y)\,dy<\infty.
\end{equation}
 Here, $\phi_j$ is a strictly increasing continuous function with $\phi_j(0)=0$ and $\phi_j(1)=1$ so that
\eqref{eq:624-1BF} holds with
$0<\beta_{1,\phi_j }\le \beta_{2,\phi_j }<2$.

We claim that
$\PHI(\phi)$ holds with $\phi(r)=\phi_j(r){\bf 1}_{\{0\le r\le 1\}}+ r^2 {\bf 1}_{\{r>1\}}$, if and only if $\UJS$ holds for the jumping kernel $J(x,y)$ given above.
Obviously $\J_{\phi,\le}$ holds. Next we
check the condition $\CSJ(\phi)$.
Indeed, for any $x_0\in \bR^d$ and $0<r\le R$, we choose a Lipschitz continuous function $\vp:\bR^d\to [0,1]$ such that $\vp(x)=1$ for all $x\in B(x_0,R)$, $\vp(x)=0$ for all $x\in B(x_0,R+r)^c$ and $\|\nabla \vp\|_\infty \le 2/r$. Then,
\begin{align*}\Gamma(\vp,\vp)(x)=&\int_{\bR^d}(\vp(x)-\vp(y))^2J(x,y)\,dy\le 4 \int_{\bR^d}\left(\frac{|x-y|^2}{r^2}\wedge1\right)J(x,y)\,dy\\
\le&\begin{cases} {c_1}{r^{-2}}\displaystyle\int_0^r\frac{s}{\phi_j(s)}\,ds+\int_r^1 \frac{1}{s\phi_j(s)}\,ds+\int_{\{|x-y|\ge1\}} J(x,y)\,dy,&\quad 0<r\le 1\\
 {c_1}{r^{-2}}\displaystyle\int_{\bR^d}|x-y|^2J(x,y)\,dy,&\quad r>1\end{cases}\\
 \le & \frac{c_2}{\phi(r)},\end{align*}
where the last inequality is due to \eqref{e:5.3} and \eqref{e:5.7}.
 $\CSJ(\phi)$ is a direct consequence of the above estimate.
We now show $\PI(\phi)$ holds.
For this, we consider the $1$-truncation of the Dirichlet from $(\sE,\sF)$, i.e.
 $$
 \sE^{(1)}(f,f)=\int_{\bR^d\times \bR^d}  (f(x)-f(y))^2J(x,y) {\bf1}_{\{|x-y|\le1\}}\,dx\,dy.
 $$
 It follows from the proofs of \cite[Theorem 3.5 and Proposition 3.8]{CKK1} that $\NDL(\phi)$ holds for $(\sE^{(1)},\sF)$.
 (In fact,
 the paper \cite{CKK1} only treats the case that $\phi_j(r)=r^\alpha$ for some $\alpha\in (0,2)$, but the extension to the general $\phi_j$ as above is easy.) Then, by \cite[Proposition 3.5(1)]{CKW2}, $\PI(\phi)$ is satisfied for $(\sE^{(1)},\sF)$. Denote by $\Gamma^{(1)}(f,f)$ and $ \Gamma(f,f)$ the
carr\'e du-Champ operator for $(\sE^{(1)},\sF)$ and $(\sE,\sF)$, respectively.
 Since $\Gamma^{(1)}(f,f)\le \Gamma(f,f)$ for all $f\in \sF$, we have $\PI(\phi)$ for $(\sE,\sF)$ too. Therefore, combining all the conclusions above,
we have by Theorem \ref{T:PHI} that $\PHI(\phi)$ holds for $(\sE, \sF)$ if and only if $\UJS$ holds
for $J(x, y)$.
 \end{example}

\section{Appendix: Proof of assertions
in Example $\ref{Exam-M}$}\label{S:6}

We begin with the following simple statement.

\begin{lemma}\label{L-1} Let $S$ be the subordinator given in Subsection $\ref{S:1.1}$. Then,
\begin{itemize}
\item[{\rm (i)}] for all $r>0$, $$f(r) \simeq
r^{\gamma_1} \wedge r.$$
\item[{\rm (ii)}] for any $t>0$, $S_t$ has a density function $\eta_t(u)$ with respect to the Lebesgue measure.
\item[{\rm (iii)}] $\bE S_1<\infty$, and so for all $t>0$, $\bE S_t=t\bE S_1<\infty$.
\end{itemize} \end{lemma}

\begin{proof}(i) It follows from
elementary calculations as shown below.
For $r>1$, by $\gamma_2> \gamma_1$,
$$f(r)\le \int_0^\infty (1-e^{-rs})\frac{1}{s^{1+\gamma_1}} \,ds\le c_1r^{\gamma_1},$$ where the last inequality follows from the standard $\gamma_1$-subordinator whose L\'evy measure is comparable with $\frac{1}{s^{1+\gamma_1}} \,ds$. On the other hand, according to the fact that $1-e^{-u}\ge e^{-1}u$ for all $u\in (0,1)$, we have
$$f(r)\ge \int_0^{1/r}(1-e^{-sr})\frac{1}{s^{1+\gamma_1}} \,ds\ge e^{-1}\int_0^{1/r} sr\frac{1}{s^{1+\gamma_1}} \,ds\ge c_2 r^{\gamma_1}.$$

For $r\in (0,1]$,
\begin{align*} f(r)&=\int_0^1 (1-e^{-rs})\frac{1}{s^{1+\gamma_1}} \,ds+\int_1^{1/r} (1-e^{-rs})\frac{1}{s^{1+\gamma_2}} \,ds+ \int_{1/r}^\infty (1-e^{-rs})\frac{1}{s^{1+\gamma_2}} \,ds\\
&\le \int_0^1 rs \frac{1}{s^{1+\gamma_1}} \,ds+\int_1^{1/r} rs \frac{1}{s^{1+\gamma_2}} \,ds+\int_{1/r}^\infty \frac{1}{s^{1+\gamma_2}} \,ds\\
&\le c_3r,\end{align*}where in the first inequality we used the fact that $1-e^{-u}\le u$ for all $u>0$, and the last inequality follows from $\gamma_2>1$.
On the other hand, for $r\in (0,1]$, we also have
$$f(r)\ge \int_0^1 (1-e^{-rs})\frac{1}{s^{1+\gamma_1}} \,ds\ge e^{-1}r\int_0^1\frac{1}{s^{1+\gamma_1}} \,ds\ge c_4r.$$

Combining with all the conclusions above, we obtain the first assertion.

(ii) This immediately follows from $$\lim_{r\to\infty} \frac{
f(r)}{r^{\gamma_1}}\in(0,\infty),$$ see e.g. \cite[Theorem 1]{KS}.

(iii) This is a consequence of $\int_0^\infty s\nu(s)\,ds=\infty$,
due to the fact that $\gamma_2>1$. See, e.g., the proof of
\cite[Proposition 1.2]{BSW}. \qed

 \end{proof}

\medskip

\noindent {\bf Proof of Claim 1
in Example $\ref{Exam-M}$.}
 (i) Let $p(t,x,y)$ be the transition density function of the process $Y$. Let $\eta_t(u)$ be the density function of $S_t$ as in Lemma \ref{L-1}(2). Then,
$$p(t,x,y)=\int_0^\infty q(u,x,y)\eta_t(u)\,du,\quad t>0, x,y\in M,$$ which yields that the jumping kernel $J(x,y)$ of the process $Y$ is given by
$$J(x,y)=\int_0^\infty q(u,x,y)\nu(u)\,du.$$ Next, we write $$J(x,y)= \int_0^{d(x,y)\vee1} q(u,x,y)\nu(u)\,du+\int_{d(x,y)\vee1}^\infty q(u,x,y)\nu(u)\,du=:I_1+I_2,$$ and consider bounds for $J(x,y)$.

First, let $x,y\in M$ with $d(x,y)\le 1$. Then, by $\gamma_2>\gamma_1$,
$$I_2\le \int_{1}^\infty q(u,x,y)\nu(u)\,du\le \int_1^\infty q(u,x,y) \frac{1}{u^{1+\gamma_1}}\,du.$$ This along with $I_1$ yields that
$$J(x,y)\le \int_0^\infty q(u,x,y) \frac{1}{u^{1+\gamma_1}}\,du \le \frac{c_1}{V(x,d(x,y))d(x,y)^{\beta\gamma_1}}=\frac{c_1}{V(x,d(x,y))d(x,y)^{\alpha_1}},$$ where the inequality above follows from the standard $\gamma_1$ subordination of the process $X$.
On the other hand, by changing the variable $s\to u=d(x,y)^\beta/s$, we find that
\begin{align*}I_1&=d(x,y)^{-\gamma_1\beta}\int_{d(x,y)^\beta}^\infty \frac{1}{V(x,d(x,y)/u^{1/\beta})} \exp\left(-cu^{1/(\beta-1)}\right) u^{-1+\gamma_1}\,du\\
&\ge c_2d(x,y)^{-\gamma_1\beta}\int_{1}^2
\frac{1}{V(x,d(x,y)/u^{1/\beta})} \,du\ge
\frac{c_3}{V(x,d(x,y))d(x,y)^{\alpha_1}},\end{align*} where the last
inequality follows from VD.

Second, let $x,y\in M$ with $d(x,y)\ge1$. Then, also by $\gamma_2>\gamma_1$,
\begin{align*}I_1&=\int_0^1q(u,x,y)\nu(u)\,du+\int_1^{d(x,y)}q(u,x,y)\nu(u)\,du\\
&\le \int_0^1 q(u,x,y)u^{-1-\gamma_2}\,du+\int_1^{d(x,y)}q(u,x,y)\nu(u)\,du,\end{align*}which along with $I_2$ and the change of variable $s\to u=d(x,y)^\beta/s$ yields that
\begin{align*}J(x,y)\le & \int_0^{\infty}q(u,x,y) \frac{1}{u^{1+\gamma_2}}\,du\\
=& d(x,y)^{-\gamma_2\beta}\int_0^\infty \frac{1}{V(x,d(x,y)/u^{1/\beta})} \exp\left(-cu^{1/(\beta-1)}\right) u^{-1+\gamma_2}\,du\\
\le &c_4d(x,y)^{-\gamma_2\beta}\int_0^1\frac{1}{V(x,d(x,y)/u^{1/\beta})}\,du\\
&+c_5d(x,y)^{-\gamma_2\beta}\int_1^\infty\frac{1}{V(x,d(x,y)/u^{1/\beta})} \exp\left(-\frac{c}{2}u^{1/(\beta-1)}\right)\,du\\
\le&\frac{c_6}{V(x,d(x,y))d(x,y)^{\gamma_2\beta}}=\frac{c_6}{V(x,d(x,y))d(x,y)^{{\alpha_2}}},\end{align*} where the last inequality also follows from VD. Similarly, it holds that
\begin{align*}I_2=&d(x,y)^{-\gamma_2\beta}\int_0^{d(x,y)^\beta}
\frac{1}{V(x,d(x,y)/u^{1/\beta})}
\exp(-cu^{1/(\beta-1)})u^{-1+\gamma_2}\,du\\
\ge&c_7d(x,y)^{-\gamma_2\beta}\int_{1/2}^1
\frac{1}{V(x,d(x,y)/u^{1/\beta})}\,du\\
\ge&\frac{c_8}{V(x,d(x,y))d(x,y)^{\gamma_2\beta}}=\frac{c_8}{V(x,d(x,y))d(x,y)^{{\alpha_2}}}.\end{align*} This completes the proof of the first assertion.

(ii) For any $x\in M$ and $r>0$,
\begin{align*}\int_{B(x,r)^c} p(t,x,y)\,\mu(dy)&=\int_{B(x,r)^c}\int_0^\infty q(u,x,y)\eta_t(u)\,du\,\mu(dy)\\
&=\int_0^\infty\eta_t(u)\int_{B(x,r)^c}q(u,x,y)\,\mu(dy)\,du\\
&\le c_1 \int_0^\infty \eta_t(u) \frac{u}{r^\beta}\,du=\frac{c_1}{r^\beta}\bE S_t=c_2\frac{t}{r^\beta},\end{align*} where the inequality concerning
the tail probability
of the diffusion $X$ follows from \eqref{Up}, see e.g.\ \cite{GH0}, and in the last equality we used Lemma \ref{L-1}(3). Therefore, we find that for any $x\in M$ and $r>0$,
\begin{equation}\label{ext:1}
\bP^x \left( \tau^Y_{B(x,r)}\le t \right)\le c_3 {t}/{r^\beta},
\end{equation}
where
$$
\tau^Y_{B(x,r)}=\inf\{t>0: Y_t\notin B(x,r)\}.
$$

Now let us consider the following jumping kernel:
 $$J^*(x,y)=\begin{cases} J(x,y),&\quad d(x,y)\le 1,\\
\frac{1}{V(x,d(x,y))d(x,y)^{{\alpha_1}}},&\quad d(x,y)\ge 1.\end{cases}$$ Clearly, $J^*(x,y)$ is comparable to the jumping kernel corresponding to the standard $\gamma_1$-subordination of the process $X$. Let $Y^*:=(Y^*_t)_{t\ge0}$ be the symmetric Markov jump process associated with
the jumping kernel $J^*(x,y)$,
and let $p^{Y^*}(t,x,y)$ be the transition density function of the process $Y^*$. Since $X$ is
conservative, $Y^*$ is also conservative.
According to the stability result
for heat kernel upper bounds
in \cite[Theorem 1.15]{CKW1}, we have
$$p^*(t,x,y)\preceq \frac{1}{V(x,t^{1/{\alpha_1}})}\wedge \frac{ t}{V(x,d(x,y)) d(x,y)^{\alpha_1}}.$$ Now, let $Y^{*,1}$ be the truncated process associated with the process $Y^*$ by removing jumps bigger than $1$; that is, the jumping kernel of the process $Y^{*,1}$ is given by
$$J^{*,1}(x,y)=J(x,y){\bf1}_{\{d(x,y)\le 1\}}.$$ By \cite[Theorem 1.15 and Lemma 4.19]{CKW1},
there are constants $c_4>0$ and $c_5\in(0,1)$ such that for all $x\in M$ and $t>0$,
$$
\bP^x\left(\tau^{Y^{*,1}}_{B(x,r)}\le t \right)\le \frac{c_4 t}{r^{\alpha_1}}+
c_5,
$$
which along with \cite[Lemma 7.8]{CKW1} (note that $Y^{*,1}$ is also the truncated process associated with the process $Y$ by removing jumps bigger than $1$) yields that for all $x\in M$ and $t,r>0$,
$$
\bP^x \left( \tau^Y_{B(x,r)}\le t \right)\le \frac{c_4 t}{r^{\alpha_1}}+ c_5 +c_6 t.
$$
In particular, there are constants $c_0, \varepsilon_0\in (0,1)$ such that for all $r\le 1$ and $x\in M$,
\begin{equation}\label{ext:2}
\bP^x \left( \tau^Y_{B(x,r)}\le c_0 r^{\alpha_1} \right)\le \varepsilon_0.
\end{equation}
Hence, according to \eqref{ext:1} and \eqref{ext:2}, there are constants $c_*, \varepsilon_*\in (0,1)$ such that for all $r>0$ and $x\in M$,
$$
\bP^x \left( \tau^Y_{B(x,r )} \le c_* (r^{\alpha_1}\vee r^\beta) \right)\le
\varepsilon_*.
$$
 This immediately yields
$$
\bE^x \left[ \tau^Y_{B(x,r)} \right]\ge c_* (r^{\alpha_1}\vee r^\beta)\bP^x
 \left( \tau^Y_{B(x,r)} \ge c_* (r^{\alpha_1}\vee r^\beta) \right) \ge
(1-\varepsilon_*)c_* (r^{\alpha_1}\vee r^\beta)).
$$
On the other hand, the estimate above along with (the proofs of) \cite[Lemma 3.4, Proposition 3.6 and Proposition 2.3(5)]{CKW1} gives us that for any $x_0\in M$ and $r>0$,
$$
{\rm cap}_Y(B(x_0,r), B(x_0,2r))
\le \begin{cases}  {c_7V(x_0,r)}/{r^{\alpha_1}},&\quad r\le 1,\\
{c_7V(x_0,r)}/{r^{\beta}},&\quad r\ge 1.
\end{cases}
$$
Note that, in the argument above we need to use the fact that for any $x\in M$ and $r>0$,
$$\int_{\{y\in M: d(x,y)\ge r\}} J(x,y)\,\mu(dy)\le \frac{c_8}{r^{{\alpha_1}} \vee r^{{\alpha_2}} } \le \frac{c_8}{r^{{\alpha_1}} \vee r^{\beta} } .$$

Next, we verify the upper bound of $\bE^x\tau^Y_{B(x,r)}$ and the lower bound
of the capacity. For $t\ge 1$ and $x,y\in M$,
\begin{align*}p(t,x,y)&\le c_1\int_0^\infty \frac{1}{V(x,u^{1/\beta})} \eta_t(u)\,du\\
&= c_1\int_0^t \frac{1}{V(x,u^{1/\beta})} \eta_t(u)\,du+\frac{c_1}{V(x,t^{1/\beta})} \int_t^\infty \eta_t(u)\,du\\
&\le \frac{c_2}{V(x,t^{1/\beta})}\int_0^t \left(\frac{t}{u}\right)^{d/\beta}\eta_t(u)\,du + \frac{c_1}{V(x,t^{1/\beta})}\\
&\le \frac{c_2 t^{d/\beta}}{V(x,t^{1/\beta})}\int_0^\infty {u}^{-d/\beta}\eta_t(u)\,du + \frac{c_1}{V(x,t^{1/\beta})}, \end{align*} where the constant $d$ comes from VD.
Let $f$ be the Bernstein function of the subordinator $S$. According to \cite[p. 298]{GRW}, we know that for all $t>0$,
$$\int_0^\infty {u}^{-d/\beta}\eta_t(u)\,du=\int_0^\infty
u^{-d/\beta-1} e^{-t f(u)}\,du
\le c_3 t^{-d/\beta},$$ where in the inequality we used the fact that
$f(r)\asymp r\wedge r^{\gamma_1}$
from Lemma \ref{L-1}(1) and the Abelian and Tauberian theorem (see \cite[Theorems 1.7.1 and 1.7.1']{BGT}). Therefore, for any $x,y\in M$ and $t\ge 1$,
$$p(t,x,y)\le \frac{C_1}{V(x,t^{1/\beta})}.$$

 Next, we consider the case that $t\le 1$ and $x,y\in M$,
\begin{align*}p(t,x,y)&\le c_4\int_0^\infty \frac{1}{V(x,u^{1/\beta})} \eta_t(u)\,du\\
&\le c_4\int_0^{t^{1/\gamma_1}} \frac{1}{V(x,u^{1/\beta})} \eta_t(u)\,du+\frac{c_4}{V(x,t^{1/(\beta\gamma_1)})} \int_{t^{1/\gamma_1}}^\infty \eta_t(u)\,du\\
&\le \frac{c_5}{V(x,t^{1/(\beta\gamma_1)})}\int_0^{t^{1/\gamma_1}} \left(\frac{t^{1/\gamma_1}}{u}\right)^{d/\beta}\eta_t(u)\,du + \frac{c_4}{V(x,t^{1/{\alpha_1}})}\\
&\le \frac{c_5 t^{d/{\alpha_1}}}{V(x,t^{1/{\alpha_1}})}
\int_0^\infty {u}^{-d/\beta}\eta_t(u)\,du + \frac{c_4}{V(x,t^{1/{\alpha_1}})}. \end{align*} Following the same argument as the case $t\ge 1$, we can apply
$f(r)\asymp r\wedge r^{\gamma_1}$
and the Abelian and Tauberian theorem again, and obtain that
$$\int_0^\infty {u}^{-d/\beta}\eta_t(u)\,du
\simeq t^{-d/{\alpha_1}},\quad t\in (0,1].$$
Hence, for all $t\in (0,1]$ and $x,y\in M$,
$$p(t,x,y)\le \frac{C_2}{V(x,t^{1/{\alpha_1}})}.$$

Putting both estimates together,
we have
for all $x,y\in M$ and $t>0$,
\begin{equation}\label{e:mottee} p(t,x,y)\le \frac{C_0}{V(x,t^{1/{\alpha_1}}\wedge t^{1/\beta})}.\end{equation}
This along with \cite[Propsition 7.3]{CKW1} yields the Faber-Krahn inequality with rate function $r^{{\alpha_1}} \vee r^{\beta}$.
Then, by \cite[Lemma 4.14]{CKW1}, for any $x\in M$ and $r>0$,
$$\bE^x\tau_{B(x,r)}\le c_5(r^{{\alpha_1}} \vee r^{\beta}).$$
On the other hand, recall that $${\rm cap}_Y(B(x_0,r), B(x_0,2r))=\inf\{ \sE_Y(\phi,\phi): \phi \mbox{ is a cut-off function for } B(x_0,r) \mbox{ and } B(x_0,2r)\}$$ and
$$\lambda_{\min}(B(x_0,2r))=\inf\left\{\frac{\sE_Y(f,f)}{\mu(f^2)}: f|_{B(x_0,2r)^c}=0\right\}.$$ It holds that
\begin{equation}\label{C:low}{\rm cap}(B(x_0,r), B(x_0,2r)) \ge V(x_0,r) \lambda_{\min}(B(x_0,2r)).\end{equation}
Therefore, the lower bound for the capacity follows from the Faber-Krahn inequality above and \eqref{C:low}.
\qed

\noindent {\bf Proof of Claim 2 in Example $\ref{Exam-M}$.}
With \eqref{e:mottee} and Claim 1 of
Example \ref{Exam-M} at hand, we know from Theorem \ref{T:1.13} that $\UHK(\phi_j,\phi_c)$ holds for the process $Y$ with $\phi(r)= r^{{\alpha_1}} \vee r^{\beta}.$
Furthermore, for any $t\ge1$ and $x,y\in M$ with $d(x,y)\le t^{1/\beta}$, and any $c_0>1$, by \eqref{Up}, Lemma \ref{L-1}(1) and \cite[Proposition 2.4]{Mi},
\begin{align*} p(t,x,y)=&\int_0^\infty q(u,x,y) \eta_t(u)\,du\ge\int_t^{c_0t} q(u,x,y)\eta_t(u)\,du\\
\ge & \frac{c_1(c_0)}{V(x,t^{1/\beta})}\bP \left( t\le S_t\le c_0t \right)
\ge\frac{c_1(c_0)}{V(x,t^{1/\beta})}.
\end{align*}
Similarly, for any $t\in (0,1]$ and $x,y\in M$ with $d(x,y)\le t^{1/{\alpha_1}}$, and any $c_0>1$,
\begin{align*}
p(t,x,y)=&\int_0^\infty q(u,x,y) \eta_t(u)\,du\ge\int_{t^{1/\gamma_1}}^{c_0t^{1/\gamma_1}} q(u,x,y)\eta_t(u)\,du\\
\ge & \frac{c_3(c_0)}{V(x,t^{1/{\alpha_1}})}\bP (t^{1/\gamma_1}\le
S_t\le c_0t^{1/\gamma_1}) \ge\frac{c_4(c_0)}{V(x,t^{1/{\alpha_1}})},
\end{align*} where in the second inequality we also used the fact that ${\alpha_1}=\gamma_1\beta$. Thus, by these two estimates above, $\NL(\phi)$ holds. Therefore, according to Theorem \ref{T:main},
$\HK_-(\phi_j,\phi_c)$ holds for the process $Y$, which is exactly \eqref{e:1.3} and \eqref{e:1.4}. If in addition $(M,d,\mu)$ is connected and satisfies the chain condition, then $\HK (\phi_j, \phi_c)$ holds, which is \eqref{e:1.5}. This establishes  Claim 2 in Example $\ref{Exam-M}$.
\qed

\bigskip

\noindent \textbf{Acknowledgements.} We thank Professor
Alexander Grigor'yan concerning the question about the heat kernel of
the subordinate
process $Y$ in Example \ref{Exam-M}, which prompted our systematic
investigation carried out in this paper. The research of Zhen-Qing
Chen is partially supported by Simons Foundation Grant 520542, a
Victor Klee Faculty Fellowship at UW, and NNSFC grant 11731009. \
The research of Takashi Kumagai is supported by JSPS KAKENHI Grant
Number JP17H01093 and by the Alexander von Humboldt Foundation.\
 The research of Jian Wang is supported by the National
Natural Science Foundation of China (No.\ 11831014), the Program for
Probability and Statistics: Theory and Application (No.\ IRTL1704)
and the Program for Innovative Research Team in Science and
Technology in Fujian Province University (IRTSTFJ).

 \vskip 0.3truein

\noindent {\bf Zhen-Qing Chen}

\smallskip \noindent
Department of Mathematics, University of Washington, Seattle,
WA 98195, USA \\
\noindent School of Mathematics and Statistics, Beijing Institute of Technology, China

\noindent
E-mail: \texttt{zqchen@uw.edu}

\medskip

\noindent {\bf Takashi Kumagai}:

\smallskip \noindent
 Research Institute for Mathematical Sciences,
Kyoto University, Kyoto 606-8502, Japan

\noindent Email: \texttt{kumagai@kurims.kyoto-u.ac.jp}

\medskip

\noindent {\bf Jian Wang}:

\smallskip \noindent
College of Mathematics and Informatics \& Fujian Key Laboratory of Mathematical Analysis and Applications (FJKLMAA), Fujian Normal
University, 350007, Fuzhou, P.R. China.

\noindent Email: \texttt{jianwang@fjnu.edu.cn}

 \end{document}